\documentclass[11pt]{amsart} 
\usepackage[english]{babel}
\usepackage{mathtools}
\usepackage{stmaryrd}
\usepackage{amssymb,amscd, graphics,color,latexsym,cancel, booktabs}   
\usepackage[linktoc=all, pagebackref, hyperindex]{hyperref}%
\hypersetup{colorlinks,  citecolor=blue,  filecolor=blue,  linkcolor=blue,  urlcolor=black}
\usepackage{multirow}
\usepackage{comment}
\usepackage{extarrows}

\usepackage{tikz-cd} 
\usepackage{tikz}
\usetikzlibrary{matrix,calc}
\usepackage{mathrsfs}
\usepackage[all]{xy}
\usepackage{amsfonts}
\usepackage[normalem]{ulem}
\usepackage{euscript}
\usepackage{amsmath}
\usepackage{xcolor}
\usepackage{ifpdf}
\usepackage{cleveref}
\newcommand\restr[2]{{
  \left.\kern-\nulldelimiterspace 
  #1 
  \vphantom{\big|} 
  \right|_{#2} 
}}

\newcommand{\lar}{\longrightarrow}

\newtheorem{Theorem}{Theorem}[section]
\newtheorem{Lemma}[Theorem]{Lemma}
\newtheorem{Corollary}[Theorem]{Corollary}
\newtheorem{Proposition}[Theorem]{Proposition}

\theoremstyle{definition}
\newtheorem{Remark}[Theorem]{Remark}
\newtheorem{Example}[Theorem]{Example}

\newtheorem{Question}[Theorem]{Question}
\def\sqr#1#2{{\vcenter{\hrule height.#2pt
		\hbox{\vrule width.#2pt height#1pt \kern#1pt
		\vrule width.#2pt}
		\hrule height.#2pt}}}
\def\phi{\varphi}

\def\VaVa{{\mathcal V}\kern-5pt {\mathcal V}}
\def\gr#1#2{{\rm gr}\, _{#1}(#2)}
\def\gr{{\rm gr}\,}
\def\hht{{\rm ht}\,}
\def\depth{{\rm depth}\,}

\def\Min{{\rm Min}\,}
\def\codim{{\rm codim}\,}

\def\ker{{\rm ker}\,}
\def\grade{{\rm grade}\,}
\def\rk{\rm rank}

\def\syz{\mbox{\rm Syz}}

\def\coker{\mbox{\rm coker}}

\def\Ext#1#2#3#4{{\rm Ext}\,^{#1}_{#2}({#3},{#4})}

\def\supp#1{{\rm Supp}\, (#1)}

\def\ini{\mbox{\rm in}}

\def\der#1{\mbox{\rm der}_k(#1)}

\def\cl#1{{\mathcal #1}}

\def\ev{{\rm ev}}
\def\phi{\varphi}

\def\hht{{\rm ht}\,}
\def\grade{{\rm grade}\,}

\def\cl#1{{\cal #1}}
\def\rk{\rm rank}

\def\ch{\text{ch}}
\newcommand{\excise}[1]{}

\def\NZQ{\mathbb}               

\def\PP{{\NZQ P}}

\def\G{{\mathcal G}}

\def\T{{\mathcal T}}

\def\opn#1#2{\def#1{\operatorname{#2}}} 
\opn\chara{char} \opn\length{\lambda} \opn\pd{pd} \opn\rk{rk}
\opn\projdim{proj\,dim} \opn\injdim{inj\,dim} \opn\rank{rank}
\opn\depth{depth} \opn\grade{grade} \opn\height{height}
\opn\embdim{emb\,dim} \opn\codim{codim}

\opn\Tr{Tr} \opn\bigrank{big\,rank}
\opn\superheight{superheight}\opn\lcm{lcm}
\opn\trdeg{tr\,deg}
	\opn\reg{reg} \opn\lreg{lreg} \opn\ini{in} \opn\lpd{lpd}
	\opn\size{size} \opn\sdepth{sdepth}
	\opn\link{link}\opn\fdepth{fdepth}\opn\lex{lex}
	\opn\tr{tr}
	\opn\type{type}
    \opn\Num{Num}   \opn\Str{Str}
	\opn\div{div} \opn\Div{Div} \opn\cl{cl} \opn\Cl{Cl}
	\opn\Spec{Spec} \opn\Supp{Supp} \opn\supp{supp} \opn\Sing{Sing}
	\opn\Ass{Ass} \opn\Min{Min}\opn\Mon{Mon}
	\opn\Ho{H}  \opn\indeg{indeg} \opn\Hs{HS} \opn\Eig{E}
	\opn\Ann{Ann} \opn\Rad{Rad} \opn\Soc{Soc}  \opn\der{Der}
	\opn\Bour{Bour} \opn\proj{Proj}
    \opn\Sym{Sym}
    \opn\Lie{Lie}
	\opn\Im{Im} \opn\Ker{Ker} \opn\Coker{Coker} \opn\Am{Am}
    \opn\im{im}
	\opn\Hom{Hom} \opn\Tor{Tor} \opn\Ext{Ext} \opn\End{End}
	\opn\Aut{Aut} \opn\id{id}
	
	\opn\nat{nat}
	\opn\pff{pf}
	\opn\Pf{Pf} \opn\GL{GL} \opn\SL{SL} \opn\mod{mod} \opn\ord{ord}
	\opn\Gin{Gin} \opn\Hilb{HP}\opn\sort{sort}
	\opn\PF{PF}\opn\Ap{Ap}\opn\sat{sat}
    \opn\PGL{PGL}
    \opn\Coh{Coh}
    \opn\op{op}
	\opn\aff{aff} \opn
	\con{conv} \opn\relint{relint} \opn\st{st}
	\opn\lk{lk} \opn\cn{cn} \opn\core{core} \opn\vol{vol}  \opn\inp{inp} \opn\nilpot{nilpot}
	\opn\link{link} \opn\star{star}\opn\lex{lex}\opn\set{set}
	\opn\width{wd}
	\opn\Fr{F}
	\opn\QF{QF}
	\opn\G{G}
    \opn\Id{Id}
	\opn\type{type}\opn\res{res}
	\opn\log{Log}
    \def\Trip{{\bf Trip}}
    \def\Set{{\bf Set}}
    \def\Sch{{\bf Sch}}
    
    \def\HilbertF{{\bf Hilb}}
    \def\Hilbert{{\rm Hilb}}
    
    \def\QuotF{{\bf Quot}}
    \def\Quot{{\rm Quot}}

	\opn\gr{gr}   
	
    \opn\sm{{\rm sm}}
	\def\pot#1#2{#1[\kern-0.28ex[#2]\kern-0.28ex]}

	\opn\dirlim{\underrightarrow{\lim}}
	\opn\inivlim{\underleftarrow{\lim}}
\begin{document}
		
		\title[Geometry of three--generated ideals on the plane]{Geometry of three--generated ideals on the plane} 
		\author{Felipe Monteiro}

        \address{
        IMECC, University of Campinas (UNICAMP),
        Campinas, SP, CEP 13083-872, Brazil
        }
        
        \address{
        Universit\'e Bourgogne Europe,
        CNRS, IMB UMR 5584,
        F-21000 Dijon, France
        }
        \email{ffcmonteiro00@gmail.com}
		
		\subjclass[2020]{Primary 13D02, 14C05;Secondary 14D20, 14J60, 14H20, 14H50}

        \keywords{
        three-generated ideals, Bourbaki degree, syzygy bundles,
        Quot schemes, Hilbert schemes, semistability,
        du Plessis--Wall bounds, gradient ideals, plane curves
        }

		\begin{abstract}
		We study the geometry of parameter spaces of homogeneous ideals in $k[x,y,z]$ minimally generated by three forms of the same degree, stratified by the degree of their base scheme and by the initial degree of their syzygies. We realize these strata through Quot and Hilbert schemes, construct universal families of Bourbaki schemes, and prove smoothness and irreducibility results for the free loci. We then compare these spaces with the loci of gradient ideals and plane curves. Using slope semistability of the associated syzygy bundles, we give a vector-bundle interpretation of the refined du Plessis--Wall bounds and obtain the same numerical restriction for arbitrary three-generated ideals. Finally, in the case of triples of cubic forms, we determine all nonempty degree strata, prove that they are smooth and irreducible, and identify their gradient loci in terms of the Bourbaki stratification of reduced quartic plane curves.
		\end{abstract}
	\maketitle
    		
    	\begingroup
            \tableofcontents
        \endgroup
		
\section*{Introduction}

Three-generated homogeneous ideals provide a natural first setting
in which the geometry of rank-two syzygy modules exhibits
nontrivial behavior. In a companion work, written jointly with M.~Jardim,
A.~Nasrollah Nejad, and Z.~Ramos~\cite{JMNRS2026}, we developed
a numerical theory for three-generated homogeneous ideals. In
particular, we studied the interaction among the initial degree of the syzygy module, the codimension-two degree of the ideal, and the
Bourbaki degree, and introduced the distinction between numerical,
structural, dimensional, and integrable gaps.

The purpose of the present work is to develop the geometric and
moduli-theoretic side of this picture in the case of $R = k[x,y,z]$, where $k$ is algebraically closed of characteristic zero. Rather than asking only which pairs of numerical invariants can occur, we study how the ideals
realizing a fixed pair sit inside their natural parameter spaces,
how these loci vary in families, and how the associated syzygy and
Bourbaki constructions interact with Hilbert schemes, Quot schemes,
and moduli spaces of sheaves.

Let \(W_d=H^0(\PP^2,\mathcal O_{\PP^2}(d))^{\oplus3}\simeq R_d^3\). We consider
the open subset
\[
U_d\subset\PP(W_d)
\]
parametrizing triples \([J]=[f_1:f_2:f_3]\) which are linearly independent and have no common divisorial factor.
For such a triple, write \(
Z(J)=V(f_1,f_2,f_3)\subset\PP^2\), \(j=\deg Z(J)\) and let \(e=\indeg(\syz(J))\) be the minimal degree of a nonzero syzygy of $J$.

The corresponding syzygy bundle $\T_J \doteq \widetilde{\syz(J)}$ sits in the short exact sequence
\[
0\longrightarrow
\T_J
\longrightarrow
\mathcal O_{\PP^2}^{\oplus3}
\longrightarrow
\mathcal I_{Z(J)}(d)
\longrightarrow0.
\]
The \emph{Bourbaki degree formula} from \cite{JMNRS2026} takes the form
\[
\Bour(J)=e(e-d)+d^2-j.
\]
Thus the two numerical invariants $(e,j)$ determine the Bourbaki
degree, and the natural geometric pieces of $U_d$ are
\[
Z_j
=
\{[J]\in U_d:\deg Z(J)=j\} \text{ and }I_e^j=\{[J]\in Z_j:\indeg(\syz(J))=e\}.
\]
A distinguished linear subspace of $\PP(W_d)$ is the gradient locus
\[
G_d
=
\{[J_F]\doteq[\partial_xF:\partial_yF:\partial_zF]:
F\in R_{d+1}\}.
\]
When $F$ is reduced and not a cone, \(
j=\deg V(J_F)=\tau(F)\), the global Tjurina number. Hence intersecting the strata $I_e^j$
with $G_d$ translates the realizability problem for arbitrary triples into the corresponding integrability problem for plane curves.

We now summarize the main results.

\medskip
\noindent
\textbf{Theorem A
(= Theorems~\ref{thm:strata},
\ref{thm:open-Z-j-quot-scheme},
\ref{thm:hilb-map-j}, and
\ref{thm:morphism-Bourbaki-scheme}).}
The loci $Z_j$ and $I_e^j$ are locally closed subschemes of
$\PP(W_d)$ and give finite stratifications of $U_d$. The stratum
$Z_j$ is naturally an open subscheme of a Quot scheme, and at a
point $[J]\in Z_j$, with $Z=Z(J)$,
\[
T_{[J]}Z_j
\simeq
\Hom(\T_J,\mathcal I_Z(d)),
\]
while \(\Ext^1(\T_J,\mathcal I_Z(d))
\) is an obstruction space.

There is also a natural morphism
\[
\Phi_j:Z_j\longrightarrow\Hilbert^j(\PP^2),
\qquad
[J]\longmapsto Z(J),
\]
whose fibers are open subsets of \(\PP\bigl(H^0(\mathcal I_Z(d))^{\oplus3}\bigr)\).

The vanishing \(
H^1(\mathcal I_Z(d))=0\) implies smoothness of both $Z_j$ and $\Phi_j$ at the corresponding
point. In particular, whenever nonempty, the strata $Z_j$ are smooth
and irreducible for $j\leq d$.

Finally, on the reduced open locus of $I_e^j$ where
$h^0(\T_J(e))$ is minimal, minimal syzygies form a projective bundle
and the Bourbaki construction gives a natural morphism
\[
\PP(\mathcal E_e^j)
\longrightarrow
\Hilbert^{\,e(e-d)+d^2-j}(\PP^2)
\]
corresponding to the association of the Bourbaki scheme from a given pair of an ideal $J$ and a minimal degree syzygy $\nu \in \syz(J)_e$.

\medskip
\noindent
\textbf{Theorem B
(= Proposition~\ref{prop:free-strata-smooth},
Corollary~\ref{cor:free-points-picture}, and
Proposition~\ref{prop:dimension-bound-free-families}).}
For
\[
1\leq e\leq\left\lfloor\frac d2\right\rfloor,
\qquad
j=d(d-e)+e^2,
\]
the stratum $I_e^j$ parametrizes the free triples with
\[
\T_J
\simeq
\mathcal O_{\PP^2}(-e)
\oplus
\mathcal O_{\PP^2}(e-d).
\]
It is smooth and irreducible, is open in $Z_j$, and its closure is
an irreducible component of $Z_j$. Moreover,
\[
\dim I_e^j
=
\begin{cases}
d^2-de+e^2+3d+3e+3,
& e<d/2,\\[1mm]
\dfrac34d^2+\dfrac92d+2,
& e=d/2.
\end{cases}
\]
The free locus need not be closed in its component: for $d=6$ we
construct a degeneration from $I_3^{27}$ to $I_2^{27}$ (Example~\ref{ex:1-par-family-free-degeneration}).

We obtain expected-dimension bounds for the intersections of the
free strata with the gradient and arrangement loci, and use them to
exhibit systematic failures of transversality. In Remark~\ref{rmk:known-free-families-of-divisors} we explore some families of known divisors and how they fit into the present framework.

\medskip
\noindent
\textbf{Theorem C
(= Lemma~\ref{lem:mu-semistability},
Proposition~\ref{prop:semistable-refined-bound}, and
Proposition~\ref{prop:morphism-gieseker}).}
The syzygy bundle $\T_J$ is slope-semistable if and only if
\[
e\geq\left\lceil\frac d2\right\rceil,
\]
and slope-stable if and only if $e>d/2$. In the semistable range,
Riemann--Roch and the vanishings
\[
H^0(\T_J(e-1))
=
H^2(\T_J(e-1))
=
0
\]
give
\[
j
\leq
d(d-e)+e^2-\binom{2e-d+1}{2}.
\]
For gradient ideals with $e>d/2$, this is precisely the refined
du Plessis--Wall bound. Thus its correction term admits a
vector-bundle interpretation in terms of slope semistability and
the cohomology of the syzygy bundle. In the stable range, the
universal syzygy family also gives a morphism to the corresponding
Gieseker moduli space.

\medskip
\noindent
\textbf{Theorem D
(= Theorems~\ref{quartic-Bour-strata} and
\ref{thm:strata-d-3}).}
For triples of cubic forms, the nonempty degree strata are precisely \(Z_0,\ldots,Z_7\). They are smooth and irreducible, with
\[
\dim Z_j=29-j,
\]
and $G_3$ meets each $Z_j$ properly. Combining this with Wall's
classification of reduced quartics and the numerical classification
of~\cite[Theorem~4.1]{JMNRS2026}, we obtain the corresponding
Bourbaki stratification.

\medskip

The results above give a geometric framework for the three levels of
realizability emphasized in \cite{JMNRS2026}:
\[
\text{numerically admissible}
\supseteq
\text{realizable by a triple}
\supseteq
\text{realizable by a gradient triple}.
\]
The strata $I_e^j$ encode the middle level, while their intersections
with $G_d$ encode the last one. From this point of view, integrable
gaps arise when a nonempty stratum $I_e^j$ fails to meet $G_d$.

The paper is organized as follows.
Section~\ref{sec:framework} recalls the numerical framework and notations from \cite{JMNRS2026}.
Section~\ref{sec:constructible} constructs the basic parameter spaces
and their stratifications.
Section~\ref{sec:local-geometry} realizes the fixed-degree strata inside Quot schemes and studies their local geometry.
Section~\ref{sec:hilbert-map} introduces the morphisms to Hilbert
schemes of points.
Section~\ref{sec:Bourbaki-map} develops the universal Bourbaki
construction and the geometry of free strata.
Section~\ref{sec:gradient-gaps} studies gradient triples, free
divisors, and integrable gaps.
Section~\ref{sec:semistable-syzygies} develops the stability approach
and its relation with the refined du Plessis--Wall bounds.
Finally, Section~\ref{sec:quartic-strata} treats triples of cubics and
the Bourbaki stratification of quartic plane curves.

Throughout the paper, $k$ is an algebraically closed field of
characteristic zero. We denote by $\Sch_k$ the category of locally
Noetherian $k$-schemes, and all functors of points are defined on
$\Sch_k^{\mathrm{op}}$ unless stated otherwise.
        
\subsection*{Acknowledgments}
The author is supported by the S\~ao Paulo Research Foundation
(FAPESP), Brazil, grant \#2021/10550-4. The author is currently a
Ph.D. student under the supervision of Marcos Jardim and Daniele
Faenzi. The author thanks Marcos Jardim, Daniele Faenzi,
Abbas Nasrollah Nejad, and Zaqueu Ramos for fruitful discussions and
suggestions concerning this work.

\section{Numerical framework and notation}\label{sec:framework}

Let $R=k[x,y,z]$ be the standard graded polynomial ring over $k$, and fix an integer $d\geq 1$. We denote by $
W_d \doteq H^0(\mathcal{O}_{\mathbb{P}^2}(d))^{\oplus 3} \simeq R_d^3$ the parameter space of triples of homogeneous polynomials $(f_1, f_2, f_3)$ of degree $d$. The associated  syzygy graded $R$-module $\syz(J)$ is adopted in the following grading convention:
$$
0 \lar \syz(J) \lar R^3 \xlongrightarrow{(f_1, f_2, f_3)} J(d) \lar 0,
$$
and we denote by $e = \indeg(\syz(J))$ the minimal degree for non-trivial syzygies of $J$.

Our main objects of study are homogeneous ideals minimally generated
by three forms of the same degree and satisfying the conditions
below, with particular emphasis on the height-two almost complete
intersection case. We follow the conventions and numerical framework
developed in~\cite{JMNRS2026}.

\begin{itemize}
    \item[(a)] The ideals $J = (f_1, f_2, f_3)$ are assumed to satisfy $\gcd(f_1, f_2, f_3) = 1$, meaning the triple is primitive in the sense that it cannot be written in the form
    $$
    (f_1, f_2, f_3) = g(g_1, g_2, g_3)
    $$ with $\deg(g) > 0$. This also implies that $\hht(J) \geq 2$, so the associated closed projective subscheme $Z = V(J) \subseteq \mathbb{P}^2$ is zero-dimensional (possibly empty).
    \item[(b)] The ideals $J=(f_1,f_2,f_3)$ are assumed to be \emph{minimally generated}; equivalently, $\{f_1,f_2,f_3\}\subset R_d$ is $k$-linearly independent, or $\syz(J)_0=0$. For ideals satisfying~$(a)$, one has
    $$
    \syz(J)_0 \neq 0 \iff \syz(J) \simeq R \oplus R(-d).
    $$
    We also call this the \emph{compressible case}, since one of the generators may be removed without changing the ideal.
\end{itemize}

For ideals satisfying~(a) and~(b), after choosing a nonzero
minimal-degree syzygy $\nu\in\syz(J)_e$, one obtains a short exact
sequence
\[
0\longrightarrow
R(-e)
\xlongrightarrow{\nu}
\syz(J)
\longrightarrow
I_\nu(e-d)
\longrightarrow0,
\]
where $I_\nu\subset R$ is a homogeneous ideal whose projective zero
scheme is zero-dimensional, possibly empty. Moreover,
$I_\nu=R$ if and only if
\[
\syz(J)\simeq R(-e)\oplus R(e-d).
\]
The degree $\deg(R/I_\nu)$ is independent of the choice of the
minimal-degree syzygy $\nu$ and is called the \emph{Bourbaki degree}
of $J$. By \cite[Theorem~1.2]{JMNRS2026},
\begin{equation}\label{eq:Bourbaki-deg-formula}
\Bour(J)=e(e-d)+d^2-\deg(R/J).
\end{equation}
Here $\deg(R/J)$ denotes the codimension-two degree: it is the usual
degree when $\height(J)=2$, and by convention it is zero when $\height(J)=3$. Equivalently,
\[
\deg(R/J)=\deg V(J),
\]
where the empty projective scheme is assigned degree zero.

The formulation \eqref{eq:Bourbaki-deg-formula} generalizes the notion for projective plane curves present in \cite{MAA}. The Bourbaki degree may be viewed as a numerical gadget to measure the failure of freeness of the syzygy module, since
$$
\Bour(J) = 0 \iff \syz(J) \simeq R(-e)\oplus R(e-d).
$$
By \cite[Theorem~2.1]{JMNRS2026}, \(\Bour(J)\leq e^2\), and consequently
\[
d(d-e)\leq \deg(R/J)\leq d(d-e)+e^2,
\]
bounds that generalize the classical du--Plessis and Wall bounds (\cite[Theorem 3.2]{CTC-Plessis}) for the global Tjurina number of a reduced projective plane curve. In particular, they introduce \emph{numerically admissible ranges}
$$
j=\deg(R/J) \in [d(d-e), d(d-e)+e^2]
$$
with $\indeg(\syz(J))=e$. Following the terminology introduced in \cite{JMNRS2026}, a pair $(e,j)$ satisfying $j \in [d(d-e), d(d-e)+e^2]$ is said to be \emph{realizable} for a given degree $d$ if there is a degree $d$ triple $J = (f_1, f_2, f_3)$ such that
$$
(\indeg(\syz(J)), \deg(R/J))=(e,j).
$$

\section{The constructible sets of fixed Bourbaki degree}\label{sec:constructible}

In this section, we establish the parameter spaces relevant to our
study. Related parameter spaces for gradient ideals were previously
considered in~\cite{Futata2023}.

Retain the notation introduced in Section~\ref{sec:framework}. For \([f]=[f_1:f_2:f_3]\in\PP(W_d)\), we write
\[
J_{[f]}\doteq(f_1,f_2,f_3)\subset R.
\]
The condition \(\gcd(f_1,f_2,f_3)=1\) is equivalent to the absence of a common divisorial component and
implies $\height(J_{[f]})\geq2$. Thus the corresponding projective
base scheme
\[
Z_{[f]}\doteq V(J_{[f]})\subset\PP^2
\]
is zero-dimensional, possibly empty.

The first result of the section is to show that the space of triples $U_d \subset \mathbb{P}(W_d)$ satisfying both conditions above is open (see Proposition~\ref{prop:parameter-space-open}). Afterwards, we describe the two closed subsets of triples of interest (see Proposition~\ref{prop:parameter-space-gradient}): the locus of \emph{gradient triples}, that is, triples $[f] = [\partial_x F : \partial_y F : \partial_z F]$ for some homogeneous polynomial $F \in R_{d+1}$, and the locus of gradient triples coming from \emph{line arrangements}, when $F = h_1 \cdot \cdots \cdot h_{d+1}$ is a product of linear forms. 

The main result of the section is Theorem~\ref{thm:strata}, where we stratify the subset $U_d$ into two locally closed stratifications: one by fixing the degree $\deg(R/J_{[f]})$, and another stratification within the latter fixing the initial degree for syzygies. We also produce a functorial description of the open subset $U_d$ and of the strata of fixed degree $\deg(R/J_{[f]})$ (see Lemma~\ref{lem:functors-of-triples}). At the end, we discuss with examples our choice of stratification and flatness notions for triples (see Section~\ref{subsec:different-flatness}).

\subsection{Spaces of three-generated ideals}

\begin{Proposition}\label{prop:parameter-space-open}
For each $d \geq 1$, the subset
$$
U_d = \{[f] \in \mathbb{P}(W_d) : \hht(J_{[f]}) \geq 2\text{ and }J_{[f]}\text{ is minimally generated}\}
$$
is a dense open subset of $\mathbb{P}(W_d)$, so it has dimension $3 \binom{d+2}{2}-1$.
\end{Proposition}

\begin{proof}
Let us consider the following two subsets:
\begin{align*}
U &\doteq \{[f] \in \mathbb{P}(W_d) : \gcd(f_1,f_2,f_3) = 1\},\\
V &\doteq \{[f] \in \mathbb{P}(W_d) : \dim_k \langle f_1, f_2, f_3 \rangle = 3\}.
\end{align*}
Since $U_d = U \cap V$, it suffices to show $U$ and $V$ are open. For each $1 \leq k \leq d$, there is a proper morphism
\begin{align*}
F_k : \mathbb{P}(R_k) \times \mathbb{P}(R_{d-k}^3) &\rightarrow \mathbb{P}(W_d)\\
( [g], [(h_1, h_2, h_3)] )&\longmapsto [(g h_1, g h_2, g h_3)]
\end{align*}
induced by the corresponding $k$-bilinear map, so its image is the closed subset inside $\mathbb{P}(W_d)$ where the three polynomials have a common factor of degree $k \geq 1$. Then
$$
U = \{[f_1: f_2: f_3] \in \mathbb{P}(W_d) : \gcd(f_i) = 1\} = \mathbb{P}(W_d) \setminus \left(\bigcup_{k=1}^{d} \text{im } F_k \right).
$$
For each $1 \leq k \leq d$, one has 
\begin{align*}
\dim(\text{im }F_k) &\leq \dim \mathbb{P}(R_k) + \dim(\mathbb{P}(R_{d-k}^3))\\
&= \binom{k+2}{2}-1 + 3 \binom{d-k+2}{2}-1,
\end{align*}
and, on the other hand, a direct computation gives
\begin{align*}
\dim(\mathbb{P}(W_d)) - \dim(\text{im }F_k) &\geq 3 \binom{d+2}{2} - \binom{k+2}{2}- 3\binom{d-k+2}{2} + 1\\
&\geq k(3d - 2k + 3)\\
&\geq k(3d-2d+3) = k(d+3) > 0,
\end{align*}
so each \(\operatorname{im}(F_k)\subset\PP(W_d)\) is a proper closed subset. Since $\PP(W_d)$ is irreducible and only
finitely many values of $k$ occur,
\[
\bigcup_{k=1}^d\operatorname{im}(F_k)
\]
is a proper closed subset. Hence $U$ is dense and open.

To treat $V$, fix a basis of $R_d$ and write $f_1,f_2,f_3$ in this
basis. Set
\[
L_d
\doteq
\left\{
[f]\in\PP(W_d):
f_1\wedge f_2\wedge f_3=0
\right\}.
\]
This is a determinantal closed subset of $\PP(W_d)$. It is proper,
since
\[
[x^d:y^d:z^d]\notin L_d.
\]
Therefore \(V=\PP(W_d)\setminus L_d\) is dense and open. Hence \(U_d=U\cap V\) is dense and open in $\PP(W_d)$.
\end{proof}

The next result describes two important families inside the parameter space $\mathbb{P}(W_d)$.

\begin{Proposition}\label{prop:parameter-space-gradient}
For each $d \geq 1$, we have:
\begin{itemize}
    \item[(a)] The locus of gradient triples
    $$
    G_d \doteq \{[f_1:f_2:f_3] \in \mathbb{P}(W_d) : \exists \ F \in R_{d+1} , \ \nabla F = (f_1, f_2, f_3)\}
    $$
    is a linear closed subset, of dimension $\binom{d+3}{2}-1$.

    \item[(b)] The locus of triples coming from line arrangements
    $$
    A_d \doteq \{[\nabla F] \in G_d : F = h_1 \cdots h_{d+1}, \ \deg(h_i) = 1\}
    $$
    is a closed subset of dimension $2d+2$.
\end{itemize}
\end{Proposition}

\begin{proof}
To show $(a)$, we use the linearity of the gradient construction to induce a morphism between projective spaces: 
\begin{align*}
\overline{\nabla}_d: \mathbb{P}(R_{d+1}) &\lar \mathbb{P}(W_d)\\
[F]&\lar [\nabla F]=[\partial_x F: \partial_y F: \partial_z F]
\end{align*}
which embeds $\mathbb{P}(R_{d+1})$ linearly in $\mathbb{P}(W_d)$, since $\nabla_d$ is an injective $k$-linear map: for $F \in R_{d+1}$ such that $\nabla F = 0$, one gets $\partial_x F = \partial_y F = \partial_z F = 0$, and from the Euler relation we obtain
$$
(d+1) F = x \partial_x F + y \partial_y F + z \partial_z F = 0,
$$
so $\ker(\nabla_d) = \{0\}$.

For $(b)$, note that $\mathbb{P}(R_1) \simeq \mathbb{P}^2$, and we may form the following commutative diagram
\begin{center}
\begin{tikzcd} (\mathbb{P}^2)^{d+1} \arrow[r, "\sigma", hook] \arrow[rd, "\varphi"'] & \mathbb{P}\big((R_1)^{\otimes (d+1)}\big) \arrow[d, dashed, "\pi"] \\ & \mathbb{P}\big(\Sym^{d+1}(R_1)\big) = \mathbb{P}(R_{d+1}) 
\end{tikzcd}
\end{center}
where $\sigma$ is the Segre embedding and $\pi$ is the rational map induced by the projection $R_1^{\otimes d+1} \twoheadrightarrow \Sym^{d+1}(R_1)$. We claim that $\varphi \doteq \pi \circ \sigma$ defines a morphism of projective varieties. Indeed, the Segre variety $\text{im }\sigma$ consists of completely decomposable tensors, and the base locus of $\pi$ is composed of non-zero tensors which have zero symmetrization: for $v = v_1 \otimes \ldots \otimes v_{d+1}$ where each $v_i \in R_1 \setminus \{0\}$, its symmetrization is
$$
\Sym(v) = \sum_{\tau \in S_{d+1}} v_{\tau(1)} \cdots v_{\tau(d+1)},
$$
which is, up to a non-zero scalar, the product of $v_1 \cdots v_{d+1}$, hence non-zero since $R$ is an integral domain. Therefore, the image of $\sigma$ is disjoint from the base locus of $\pi$, and $\varphi$ is a morphism. The composition
$$
(\mathbb{P}^2)^{d+1} \xlongrightarrow{\varphi} \mathbb{P}(R_{d+1}) \xlongrightarrow{\overline{\nabla}_d} \mathbb{P}(W_d)
$$
has image $A_d$, hence it is a closed subscheme. Moreover, the fibers of the composition above are finite, given by permutations (indeed, the map $\varphi$ descends to $\Sym^{d+1}(\mathbb{P}^2)$), so that
$$
\dim(A_d) = \dim((\mathbb{P}^2)^{d+1}) = 2d+2,
$$
as claimed.
\end{proof}

In the literature, a \emph{line arrangement} usually refers to an element of the intersection $U_d \cap A_d$, that is, a completely decomposable polynomial
$$
F = h_1 \cdot \cdots \cdot h_{d+1}, \quad \deg(h_i) = 1,
$$
which is assumed to be reduced. 

The next result describes stratifications in $U_d$ related to the Bourbaki construction.

\begin{Theorem}\label{thm:strata}
For $d \geq 1$, the following hold:

\begin{enumerate}
\item[(i)] For every $j\geq 0$, the subsets 
$$
Z_j = \{[f] = [f_1:f_2:f_3] \in U_d : \deg(R/J_{[f]}) = j\}
$$
are locally closed for $j \geq 1$ in $\mathbb{P}(W_d)$, open for $j = 0$, and $U_d = \bigsqcup_{j \geq 0} Z_j$.

\item[(ii)] The subsets
$$
I_e^j \doteq \{[f] = [f_1 : f_2 : f_3] \in Z_j : \indeg(\syz(J_{[f]})) = e\}
$$
are locally closed inside $\mathbb{P}(W_d)$ and $Z_j = \bigsqcup_{1 \leq e \leq d} I_e^j$. Within these strata, the Bourbaki degree is constant. 

\item[(iii)] For $b\geq0$ and $j\geq0$, set
\begin{align*}
B_{d,b,j}
&\doteq
\left\{
[f]\in Z_j:
\Bour(J_{[f]})=b
\right\},\\
B_{d,b,j}^{(e)}
&\doteq
\left\{
[f]\in Z_j:
\Bour(J_{[f]})=b,\ 
\indeg(\syz(J_{[f]}))=e
\right\}.
\end{align*}
Then
\[
B_{d,b}
\doteq
\left\{
[f]\in U_d:\Bour(J_{[f]})=b
\right\}
=
\bigsqcup_{j=0}^{d^2}
\bigsqcup_{e=1}^{d}
B_{d,b,j}^{(e)}.
\]
In particular, $B_{d,b}$ is constructible in $\PP(W_d)$.
\end{enumerate}

The same statements hold after restricting to gradient triples or
line arrangements, by intersecting respectively with
$G_d\subset\PP(W_d)$ or $A_d\subset\PP(W_d)$.
\end{Theorem}

\begin{proof}
Let $X_d \doteq \mathbb{P}^2 \times U_d \subseteq \mathbb{P}^2 \times \mathbb{P}(W_d)$ and denote the projections by:
\begin{center}
\begin{tikzcd}
             & X_d \arrow[ld, "p"'] \arrow[rd, "q"] &     \\
\mathbb{P}^2 &                                      & U_d
\end{tikzcd}
\end{center}
As usual, we denote:
$$
\mathcal{E} \boxtimes \mathcal{F} \doteq p^*( \mathcal{E}) \otimes q^*( \mathcal{F} ),
$$
for $\mathcal{E} \in \Coh(\mathbb{P}^2)$ and $\mathcal{F} \in \Coh(U_d)$. There is a morphism of sheaves on $X_d$ given by
$$
\sigma : \mathcal{O}_{\mathbb{P}^2}^{\oplus 3} \boxtimes \mathcal{O}_{U_d} \rightarrow \mathcal{O}_{\mathbb{P}^2}(d) \boxtimes \mathcal{O}_{U_d}(1) \doteq \mathcal{L},
$$
which is the expansion over a basis of $W_d$ given by monomials in $x,y,z$ of the map 
$
(f_1, f_2, f_3): \mathcal{O}_{\mathbb{P}^2}^{\oplus 3} \lar \mathcal{O}_{\mathbb{P}^2}(d)
$ with varying $f_1, f_2$ and $f_3$. This morphism has a cokernel sheaf
$$
\mathcal{O}_{\mathbb{P}^2}^{\oplus 3} \boxtimes \mathcal{O}_{U_d} \xlongrightarrow{\sigma} \mathcal{O}_{\mathbb{P}^2}(d) \boxtimes \mathcal{O}_{U_d}(1) \lar \mathcal{O}_T(d,1) \doteq \mathcal{Q} \lar 0
$$
defined by a closed subscheme $T \subset X_d$. For each $[f] \in U_d$, we denote the restriction of the morphism $\sigma$ to $P^2 \times \{[f]\} \simeq P^2$ by $\sigma_{[f]}$, which corresponds to 
$$
\sigma_{[f]} = (f_1,f_2,f_3):\mathcal{O}_{\mathbb{P}^2}^{\oplus 3} \lar \mathcal{O}_{\mathbb{P}^2}(d),
$$
hence
$$
\restr {\coker(\sigma)} {\mathbb{P}^2 \times {[f]}} \simeq \coker(\restr {\sigma} {\mathbb{P}^2 \times \Spec k}) \simeq \mathcal{O}_{Z[f]}(d),
$$
where $Z[f] = V(J_{[f]}) \subset \mathbb{P}^2$.

Considering the flattening stratification (see, for example, \cite[Chapter 8]{mumford1966lectures}) for the coherent sheaf $\mathcal{Q}$ over $X_d \twoheadrightarrow U_d$, we obtain a finite set $\{Z_j \subset U_d\}$ of locally closed subschemes satisfying:
\begin{itemize}
    \item[(a)] For each $[f] \in Z_j$, the Hilbert polynomial of each fiber 
    $$
    \Hilb(\restr {\mathcal{O}_T(d,1)} {\mathbb{P}^2 \times [f]}, t) = \Hilb(\mathcal{O}_{Z[f]}, t) = j
    $$
    is constant;
    \item[(b)] Each restriction $\restr {\mathcal{O}_T(d,1)} {\mathbb{P}^2 \times Z_j}$ is flat over $Z_j$;
    \item[(c)] $U_d = \bigsqcup_{j \geq 0} Z_j$.
\end{itemize}
This shows ${\rm (i)}$.

For ${\rm (ii)}$, let us fix $j \geq 0$. For the inclusion $i : S \doteq Z_j \hookrightarrow U_d$ there is an associated base-change map
$$
\iota = \id_{\mathbb{P}^2} \times i : \mathbb{P}^2_S \lar \mathbb{P}^2 \times U_d = X_d,
$$
so we define
$$
\mathcal{L}_S \doteq \iota^* \mathcal{L}, \quad \sigma_S \doteq \iota^*(\sigma), \quad \mathcal{Q}_S \doteq \iota^*(\mathcal{Q}).
$$
Right-exactness of pullback gives an exact sequence
$$
\mathcal{O}_{\mathbb{P}^2_S}^{\oplus 3} \xlongrightarrow{\sigma_S} \mathcal{L}_S \lar \mathcal{Q}_S \lar 0,
$$
and we may define two kernel sheaves on $\mathbb{P}^2_S$: 
$$
\mathcal{F}_S \doteq \ker( \mathcal{L}_S \lar \mathcal{Q}_S ) = \text{im }(\sigma_S), \quad \T_S \doteq \ker(\sigma_S),
$$
so we get the two short exact sequences:
\begin{align*}
0 \lar \mathcal{F}_S \lar &\mathcal{L}_S \lar \mathcal{Q}_S \lar 0\\
0 \lar \mathcal{T}_S \lar &\mathcal{O}_{\mathbb{P}^2_S}^{\oplus 3} \lar \mathcal{F}_S \lar 0
\end{align*}
By construction, the restriction $\mathcal{Q}_S$ is flat over $S$ and $\mathcal{L}_S$ is invertible (hence flat) over $S$. Thus, the flatness of $\mathcal{F}_S$ over $S$ follows. Moreover, by repeating the same argument, we get flatness of $\mathcal{T}_S$ over $S$. 

The sheaf $\mathcal{T}_S$ will be called the \emph{universal syzygy sheaf} over the strata $Z_j$. For each triple $[f] = [f_1:f_2:f_3] \in Z_j$, the restriction to $\PP^2 \times [f] \simeq \PP^2$ induces a short exact sequence
$$
0 \lar (\mathcal{T}_S)_{[f]} \lar \mathcal{O}_{\mathbb{P}^2}^{\oplus 3} \xlongrightarrow{(f_1, f_2, f_3)} \mathcal{I}_{Z[f]}(d) \lar 0,
$$
by flatness, presenting the fiber $\mathcal{T}_{[f]}$ as the syzygy sheaf $\T_{J_{[f]}}$. Each fiber has Hilbert polynomial
\begin{align*}
\Hilb(\restr {\T_S} {\mathbb{P}^2 \times [f]}, t) &= \Hilb(\mathcal{O}_{\mathbb{P}^2}^{\oplus 3}, t) - \Hilb(\mathcal{I}_{Z[f]}(d), t) \\
&= 3\binom{t+2}{2}-\binom{t+d+2}{2}+j.  
\end{align*}

From the upper semi-continuity of cohomology for flat families (see, for example, \cite[Chapter III, Theorem 12.8]{hartshorne2013algebraic}), we conclude that
$$
Y_{e, c} \doteq \{[f] \in Z_j : h^0(\mathcal{T}_{[f]}(e)) \leq c\}
$$
is open inside $Z_j$ for every $c, e \in \mathbb{Z}_{\geq 0}$, where $\T_{[f]}(e) \doteq \restr {\T_S(e)} {\mathbb{P}^2 \times \{[f]\}}$. Then, we have, for any $1 \leq e \leq d$, the following description for the underlying set 
\begin{align*}
|I_{e}^j| &\doteq \{[f] \in Z_j : \indeg(\mathcal{T}_{[f]}) = e\}\\
&= Y_{e-1, 0} \cap (Z_j \setminus Y_{e,0})
\end{align*}
which is locally closed in $Z_j$, and thus in $\mathbb{P}(W_d)$. Then, the Bourbaki degree formula \eqref{eq:Bourbaki-deg-formula} shows that, for $[f] = [f_1:f_2:f_3] \in |I_e^j|$, the Bourbaki degree is constant, so ${\rm (ii)}$ follows. 

For ${\rm (iii)}$, if one fixes the Bourbaki degree within $Z_j$, there are at most two possible solutions $e_1, e_2 \geq 0$ for $\Bour(J) = b$, and thus
$$
\{[f] \in Z_j : \Bour(J_{[f]}) = b \} = I_{e_1}^j \cup I_{e_2}^j
$$
is a finite union of locally closed subsets, hence constructible. Since $G_d, A_d \subset \mathbb{P}(W_d)$ are closed subsets, all the considered properties remain true after intersecting with them.
\end{proof}

We endow the sets $|I_e^j|$ described above with a determinantal scheme structure, as follows. For each $l\ge0$ and $[f]\in Z_j$, consider the multiplication map
\[
\mu_{l,[f]}:R_l^{\oplus3}\longrightarrow R_{d+l},
\qquad
(a_1,a_2,a_3)\longmapsto\sum_i a_if_i.
\]
Then $\ker(\mu_{l, [f]}) = \syz(J_{[f]})_l$ is the $l$-th graded piece of the syzygy module of $J$, and the rank of the source is
$$
a_l \doteq 3 \dim_k R_{l} = 3\binom{l+2}{2}
$$
Allowing the triple $[f] = [f_1:f_2:f_3]$ to vary along $S = Z_j \subset \mathbb{P}(W_d)$, we may assemble the multiplication maps into a map of vector bundles
$$
\mu_l : R_l^{\oplus 3} \otimes \mathcal{O}_S \lar R_{d+l} \otimes_k \mathcal{O}_S(1).
$$
Another way to obtain the morphism $\mu_l$ is to twist the universal morphism over $S$
$$
\sigma_S : \mathcal{O}_{\mathbb{P}^2_S}^{\oplus 3} \lar \mathcal{O}_{\mathbb{P}^2}(d) \boxtimes \mathcal{O}_S(1)
$$
by $p_S^*\mathcal{O}_{\mathbb{P}^2}(l)$ and then push forward along $q_S:\mathbb{P}^2_S\lar S$. At geometric points $s = [J] \in S$, we have
$$
h^0(\T_{J}(l)) = \dim_k \syz(J)_l = 3 \binom{l+2}{2}- \rank(\mu_{l}(s)),
$$
hence the initial degree strata $|I_e^j|$ may be described by the scheme structure induced by the fiber product over $Z_j$ of:
$$
I_e^j \doteq Y_{e-1, 0} \times_{Z_j} D_{a_e-1}(\mu_e),
$$
where $D_{a_e-1}(\mu_e)$ is zero locus of the maximal minors of the degree-$e$ multiplication map $\mu_e$.

The next Lemma, describing the behavior of the initial degree function, will be useful in the next sections.

\begin{Lemma}\label{lem:indeg-lowersemicontinuous}
For $d \geq 1$ and $0 \leq j \leq d^2$, the function
\begin{align*}
\indeg:Z_j &\lar \mathbb{Z}\\
[f]&\longmapsto \indeg(\T_{[f]})
\end{align*}
is lower-semicontinuous.
\end{Lemma}

\begin{proof}
For each $e \geq 1$, the subset:
\[
\{[f] \in Z_j : \indeg(\mathcal{T}_{[f]}) > e\} = \bigcap_{l=1}^{e} \{{[f]} \in Z_j : h^0(\mathcal{T}_{[f]}(l)) \leq 0\}
\]
is a finite intersection of open subsets, using upper semi-continuity of $h^0$ and flatness of $\mathcal{T}$ over $Z_j$. Thus, the complement, i.e. the set where $\indeg(\T_{[f]}) \leq e$, is closed. 
\end{proof}

\subsection{Functorial descriptions}

To achieve our goals in this paper, we need a functorial description for the inclusions $Z_j \subset U_d \subset \mathbb{P}(W_d)$.

\begin{Lemma}\label{lem:functors-of-triples}
Let $d \geq 1$. For each $j \geq 0$, there is a chain of functors $\Sch_k^{\op} \lar \Set$:
$$
\Trip_{d,j} \hookrightarrow \Trip_{d}^{\circ} \hookrightarrow \Trip_d 
$$
representable by the chain of inclusions of schemes $Z_j \hookrightarrow U_d \hookrightarrow \mathbb{P}(W_d)$.
\end{Lemma}
\begin{proof}
The functor of points of $\PP(W_d)$ (see, for example, \cite[Chapter II, Theorem 7.1]{hartshorne2013algebraic}) associates to a $k$-scheme $T$
the set of equivalence classes of pairs $(\mathcal L,\mathbf f)$,
where $\mathcal L$ is an invertible sheaf on $T$ and
\[
\mathbf f:
\mathcal O_T\otimes_k W_d^\vee
\twoheadrightarrow
\mathcal L
\]
is a quotient. Two pairs $(\mathcal L,\mathbf f)$ and
$(\mathcal L',\mathbf f')$ are equivalent if there is an
isomorphism $\alpha:\mathcal L\xlongrightarrow{\sim}\mathcal L'$ such that \(\mathbf f'=\alpha\circ\mathbf f.\) For each morphism of $k$-schemes, $h: S \lar T$ and $(\mathcal{L}, \textbf{f})$ a pair over $T$, the functor assigns the pair $(h^*\mathcal{L}, h^*\textbf{f})$.

Through the isomorphisms, 
$$
\Hom(\mathcal{O}_T \otimes W_d^*, \mathcal{L}) \simeq \Hom(\mathcal{O}_T, \mathcal{L} \otimes W_d) \simeq \Hom(\mathcal{O}_T, \mathcal{L} \otimes R_d)^{\oplus 3},
$$
each morphism $\mathbf f$ is
equivalently given by three sections
\(
\mathbf f_i\in H^0(T,\mathcal L\otimes_kR_d)\), for $i = 1, 2, 3$.

We first describe the open subfunctor $\Trip_d^\circ$ represented by
$U_d$. For each pair $[\mathcal{L}, \textbf{f}] \in \Trip_d(T)$, we ask two additional assumptions:
\begin{itemize}
    \item[(a)] the following non-vanishing condition
    \begin{equation}\label{eq:condition-a-functor-of-triples}
    w_{[\textbf{f}]} = \textbf{f}_1 \wedge \textbf{f}_2 \wedge \textbf{f}_3 \in H^0\left(T, \mathcal{L}^{\otimes 3} \otimes \bigwedge^3 R_d \right) \text{ is nowhere vanishing on }T.    
    \end{equation}
    If $[\mathcal{L}, \textbf{f}] = [\mathcal{L}', \textbf{f}'] \text{ via } \alpha : \mathcal{L} \simeq \mathcal{L}'$, then 
    $$
    (\alpha^{\otimes 3} \otimes \id) w_{
    [\textbf{f}]} = w_{\textbf{f}'},
    $$
    so \eqref{eq:condition-a-functor-of-triples} is well-defined on equivalence classes. To see that it is preserved under pullbacks, let $h: S \lar T$ be a morphism of $k$-schemes. Then, the pullback
    $$
    h^* : H^0\left(T, \mathcal{L}^{\otimes 3} \otimes \bigwedge^3 R_d \right) \lar H^0\left(S, h^*\mathcal{L}^{\otimes 3} \otimes \bigwedge^3 R_d\right)
    $$
    satisfies
    $$
    h^*(w_{[f]}) = h^*(\textbf{f}_1 \wedge \textbf{f}_2 \wedge \textbf{f}_3) = h^*(\textbf{f}_1) \wedge h^*(\textbf{f}_2) \wedge h^*(\textbf{f}_3) = w_{h^*\textbf{f}},
    $$
    hence the pullbacks preserve the condition \eqref{eq:condition-a-functor-of-triples}.
    \item[(b)] Over $X = \mathbb{P}^2_T$ with projections $\mathbb{P}^2_T \xlongrightarrow{p} \mathbb{P}^2$ and $\mathbb{P}^2_T \xlongrightarrow{q} T$, let
    $$
    \sigma_i : \mathcal{O}_{X} \xlongrightarrow{q^*\textbf{f}_i} q^* \mathcal{L} \otimes_k R_d \xlongrightarrow{\id \boxtimes \text{ev}} q^* \mathcal{L} \otimes_{\mathcal{O}_X} p^* \mathcal{O}_{\mathbb{P}^2}(d) = \mathcal{L} \boxtimes \mathcal{O}_{\mathbb{P}^2}(d)
    $$
    for $i =1, 2, 3$. Joining the dual morphisms $\sigma_i^{\vee} : \mathcal{L}^\vee \boxtimes \mathcal{O}_{\mathbb{P}^2}(-d) \lar \mathcal{O}_X$ into a single map as below:
    $$
    \tau = \tau_{(\mathcal{L}, \textbf{f})} \doteq (\sigma_1^\vee, \sigma_2^\vee, \sigma_3^\vee) : (\mathcal{L}^\vee \boxtimes \mathcal{O}_{\mathbb{P}^2}(-d))^{\oplus 3} \lar \mathcal{O}_X.
    $$
    The image sheaf $\text{im }\tau = \mathcal{I}_{\mathcal{Z}_{(\mathcal{L}, \textbf{f})}} \hookrightarrow \mathcal{O}_X$ determines a closed subscheme $\mathcal{Z}_{(\mathcal{L}, \textbf{f})}$. This construction depends only on the equivalence class of
$(\mathcal L,\mathbf f)$. Indeed, if
$\alpha:\mathcal L\xlongrightarrow{\sim}\mathcal L'$ identifies the two
triples, then the corresponding morphisms $\tau$ and $\tau'$ differ
only by the induced isomorphism of their sources, and hence have the
same image ideal.

We next verify compatibility with arbitrary base change. Let
$h:S\lar T$ be a morphism of $k$-schemes and set
\[
g\doteq\id_{\PP^2}\times h:
\PP^2_S\longrightarrow\PP^2_T,
\]
and \(\widetilde{\mathcal L}
\doteq
\mathcal L\boxtimes\mathcal O_{\PP^2}(d)\). By construction of the evaluation maps,
\[
g^*\tau_{(\mathcal L,\mathbf f)}
=
\tau_{h^*(\mathcal L,\mathbf f)}.
\]
We emphasize that pullback does not, in general, commute with taking
images. Instead, factor $\tau_{(\mathcal L,\mathbf f)}$ as
\[
(\widetilde{\mathcal L}^{\vee})^{\oplus3}
\twoheadrightarrow
\mathcal I_{\mathcal Z_{(\mathcal L,\mathbf f)}}
\hookrightarrow
\mathcal O_{\PP^2_T}.
\]
Since pullback is right exact, after applying $g^*$ we obtain a
surjection
\[
g^*(\widetilde{\mathcal L}^{\vee})^{\oplus3}
\twoheadrightarrow
g^*\mathcal I_{\mathcal Z_{(\mathcal L,\mathbf f)}}.
\]
Therefore
\[
\im\bigl(g^*\tau_{(\mathcal L,\mathbf f)}\bigr)
=
\im\left(
g^*\mathcal I_{\mathcal Z_{(\mathcal L,\mathbf f)}}
\longrightarrow
\mathcal O_{\PP^2_S}
\right).
\]
The right-hand side is precisely the ideal sheaf of the
scheme-theoretic inverse image
\[
\mathcal Z_{(\mathcal L,\mathbf f)}\times_T S,
\]
whereas the left-hand side is \(
\mathcal I_{\mathcal Z_{h^*(\mathcal L,\mathbf f)}}\). Hence
\begin{equation}\label{eq:pullback-Z-scheme}
\mathcal Z_{h^*(\mathcal L,\mathbf f)}
=
\mathcal Z_{(\mathcal L,\mathbf f)}\times_T S.
\end{equation}

In particular, for every geometric point $t\lar T$,
\[
\bigl(\mathcal Z_{(\mathcal L,\mathbf f)}\bigr)_t
=
V\bigl(
\mathbf f_1(t),\mathbf f_2(t),\mathbf f_3(t)
\bigr)
\subseteq\PP^2_{\kappa(t)}
\]
scheme-theoretically. The second condition defining $\Trip_d^\circ$ is that these geometric
fibers be finite, possibly empty. Since the three forms are linearly
independent, this is equivalent to
\[
\height\bigl(
\mathbf f_1(t),\mathbf f_2(t),\mathbf f_3(t)
\bigr)\geq2
\]
for every geometric $t\lar T$, or equivalently to saying that the
three forms have no common divisorial factor.
\end{itemize}

These two fiberwise conditions are exactly the conditions defining the open subset
\[
U_d\subseteq\PP(W_d).
\]
Thus a morphism $T\lar\PP(W_d)$ factors through $U_d$ if and only if
the corresponding pair $(\mathcal L,\mathbf f)$ satisfies the two
conditions above. Hence $U_d$ represents $\Trip_d^\circ$.

Finally, let $Z_j\subseteq U_d$ be the flattening stratum of the
universal base scheme corresponding to the constant Hilbert
polynomial $j$. By the universal property of the flattening
stratification, a morphism $T\lar U_d$ factors through $Z_j$ if and
only if the associated closed subscheme
\[
\mathcal Z_{(\mathcal L,\mathbf f)}
\subseteq\PP^2_T
\]
is $T$-flat and every geometric fiber has length $j$. Therefore, the subfunctor
\[
\Trip_{d,j}(T)
=
\left\{
[\mathcal L,\mathbf f]\in\Trip_d^\circ(T)
\ \middle|\
\begin{array}{l}
\mathcal O_{\mathcal Z_{(\mathcal L,\mathbf f)}}\text{ is flat over }T,\\[2mm]
\length_{\kappa(t)}
\bigl(\mathcal Z_{(\mathcal L,\mathbf f),t}\bigr)=j
\text{ for every geometric }t\lar T
\end{array}
\right\},
\]
is represented by $Z_j$. Consequently, the inclusions of functors
\[
\Trip_{d,j}\hookrightarrow
\Trip_d^\circ\hookrightarrow
\Trip_d
\]
are represented by the inclusions \(Z_j\hookrightarrow U_d\hookrightarrow\PP(W_d)\).
\end{proof}

\subsection{Different notions of flat families}
\label{subsec:different-flatness}

There are two natural notions of flatness for a family of triples.
Let \([\mathcal L,\mathbf f]\in\Trip_d(T)\). Locally on $T$, after trivializing $\mathcal L$, the three sections
define a homogeneous ideal
\[
J=(F_1,F_2,F_3)\subseteq\mathcal O_T[x,y,z],
\]
and these local ideals glue to a quasi-coherent graded ideal
$\mathcal J_{[\mathcal L,\mathbf f]}$. Set
\[
\mathcal Q_{[\mathcal L,\mathbf f]}
\doteq
\frac{\mathcal O_T[x,y,z]}
     {\mathcal J_{[\mathcal L,\mathbf f]}}.
\]
We distinguish:

\begin{itemize}
\item[(a)] \emph{quotient-flatness}: $\mathcal Q_{[\mathcal L,\mathbf f]}$
is $T$-flat;

\item[(b)] \emph{subscheme-flatness}: the projective base scheme \(\mathcal Z_{[\mathcal L,\mathbf f]}
\subseteq\PP^2_T\) is $T$-flat.
\end{itemize}

Throughout this paper we use subscheme-flatness. For gradient ideals
of reduced plane curves, this amounts to keeping the global Tjurina
number constant, which is weaker than equisingularity.

There is a practical reason for this choice. The universal graded
quotient over $\PP(W_d)$ is not coherent as an
$\mathcal O_{\PP(W_d)}$-module. For instance, at
\[
[x^d:x^{d-1}y:y^d]
\]
its fiber contains the infinite linearly independent family \(\{1,z,z^2,\ldots\}.\)
Thus ordinary flattening stratification is naturally applied instead
to the projective base schemes.

\begin{Lemma}\label{lem:quot-flatness-implies-subscheme-flatness}
Quotient-flatness implies subscheme-flatness.
\end{Lemma}

\begin{proof}
The assertion is local on $T$. Write
\[
T=\Spec A,
\qquad
Q=A[x,y,z]/(F_1,F_2,F_3),
\]
and assume that $Q$ is $A$-flat. On a standard affine chart of
$\proj(Q)$ one has
\[
D_+(x_i)\cap\proj(Q)
=
\Spec\bigl((Q_{x_i})_0\bigr).
\]
The localization $Q_{x_i}$ is $A$-flat, and its degree-zero
part is a direct summand as an $A$-module. Hence
$(Q_{x_i})_0$ is $A$-flat. The standard charts cover
$\proj(Q)$, proving the claim.
\end{proof}

The converse fails because passage to $\proj$ only remembers the
saturation.

\begin{Example}\label{ex:subscheme-flat-and-quot-flat}
Consider
\[
F_t=x^4+t x^2y^2+y^3z,
\]
whose gradient ideal is
\[
J_t=
(4x^3+2txy^2,\,
 2tx^2y+3y^2z,\,
 y^3).
\]
For every $t$, the projective scheme $V(J_t)$ is supported at $[0:0:1]$ and has length six. Thus the fibers of the projective family have constant Hilbert polynomial. By \cite[Chapter~III, Theorem~9.9]{hartshorne2013algebraic},
the corresponding family over $\mathbb A^1$ is flat.

On the other hand, the graded quotient
\[
Q=k[x,y,z,t]/J_t
\]
is not flat over $k[t]$. Indeed, \(H=x^2y^2\notin J_t
\), whereas
\[
2tH
=
y(2tx^2y+3y^2z)-3z(y^3)
\in J_t.
\]
Thus the class of $H$ is nonzero $t$-torsion in $Q$.

This torsion disappears after saturation. At $t=0$,
\[
J_0=(4x^3,3y^2z,y^3),
\qquad
y^2\in J_0^{\mathrm{sat}}\setminus J_0,
\]
and therefore \(x^2y^2\in J_0^{\mathrm{sat}}\). This explains geometrically why subscheme-flatness may hold although
quotient-flatness fails.
\end{Example}

\section{Local geometry of the strata of fixed degree}\label{sec:local-geometry}

This section is devoted to the study of the strata $Z_j$ described in Theorem~\ref{thm:strata}. We start with some preliminary lemmas about the functor of triples.

\begin{Lemma}
\label{lem:compatibility-determinant}
Let $T$ be a locally Noetherian $k$-scheme and let \([\mathcal L,\mathbf f]\in \Trip_{d,j}(T)\). Write
\[
p:\PP^2_T\lar\PP^2,
\qquad
q:\PP^2_T\lar T
\]
for projections and set \( \widetilde{\mathcal L}
\doteq
q^*\mathcal L\otimes p^*\mathcal O_{\PP^2}(d).
\)
Let \(
\lambda:
\mathcal O_{\PP^2_T}^{\oplus3}
\longrightarrow
\widetilde{\mathcal L}\) be the evaluation morphism determined by $\mathbf f$, let
$\mathcal Z\subset\PP^2_T$ be its base scheme, and put
\[
\mathcal F\doteq\im(\lambda)
=
\mathcal I_{\mathcal Z}\otimes\widetilde{\mathcal L},
\qquad
\mathcal K\doteq\ker(\lambda).
\]
Then:

\begin{enumerate}
    \item[(a)] $\mathcal K$ is locally free of rank two.

    \item[(b)] The canonical wedge morphism
    \[
    \delta:
    \mathcal O_{\PP^2_T}^{\oplus3}
    \longrightarrow
    \det(\mathcal K)^\vee
    \]
    annihilates $\mathcal K$ and hence factors through $\mathcal F$.
    Moreover, there is an isomorphism
    \[
    \theta:
    \det(\mathcal K)^\vee
    \xlongrightarrow{\sim}
    \widetilde{\mathcal L}
    \]
    such that
    \(
    \theta\circ\delta=\lambda.
    \)

    \item[(c)] The constructions of $\mathcal K$, $\delta$, and $\theta$
    are compatible with arbitrary base change $S\lar T$.
\end{enumerate}

In particular, the determinant reconstruction applied to the quotient
associated with $[\mathcal L,\mathbf f]$ recovers the original triple
up to the equivalence defining $\Trip_{d,j}$.
\end{Lemma}

\begin{proof}
Since $\mathcal Z\lar T$ is flat, the sequence
\[
0\lar\mathcal I_{\mathcal Z}
\lar\mathcal O_{\PP^2_T}
\lar\mathcal O_{\mathcal Z}
\lar0
\]
shows that $\mathcal I_{\mathcal Z}$ is $T$-flat. Hence \(\mathcal F
=\mathcal I_{\mathcal Z}\otimes\widetilde{\mathcal L}
\)
is $T$-flat, and from the short exact sequence
\[
0\lar\mathcal K
\lar\mathcal O_{\PP^2_T}^{\oplus3}
\lar\mathcal F
\lar0
\]
it follows that $\mathcal K$ is $T$-flat and its formation commutes
with base change.

For every geometric point $t\lar T$, the sheaf $\mathcal F_t$ is
torsion-free of rank one on the regular surface $\PP^2_t$. The depth
lemma therefore shows that $\mathcal K_t$ is locally free of rank
two. By the fiberwise criterion for local freeness, $\mathcal K$ is
locally free of rank two. This proves~(a).

Let
\(
\iota:\mathcal K\hookrightarrow
\mathcal V\doteq\mathcal O_{\PP^2_T}^{\oplus3}\) denote the injection. The wedge pairing induces canonically
\[
\det(\mathcal K)\otimes\mathcal V
\xlongrightarrow{\;\wedge^2(\iota)\otimes\id\;}
\bigwedge^2\mathcal V\otimes\mathcal V
\xlongrightarrow{\;\wedge\;}
\bigwedge^3\mathcal V
=
\det(\mathcal V) \simeq \mathcal{O}_{\mathbb{P}^2_T},
\]
or equivalently
\[
\delta:
\mathcal V\longrightarrow\det(\mathcal K)^\vee.
\]
Since
\[
\bigwedge^3\mathcal K=0,
\]
the morphism $\delta$ vanishes on $\mathcal K$ and therefore factors
through $\mathcal F$.

It remains to compare $\delta$ with $\lambda$. Work locally at
$x\in\PP^2_T$ and trivialize $\mathcal K$ and
$\widetilde{\mathcal L}$. We obtain
\[
0\lar R^{\oplus2}
\xlongrightarrow{A}
R^{\oplus3}
\xlongrightarrow{(f_1,f_2,f_3)}
I
\lar0,
\qquad
R=\mathcal O_{\PP^2_T,x}.
\]
By Hilbert--Burch (\cite[Theorem~20.15]{EisenbudBook}),
\[
(f_1,f_2,f_3)
=
a(\Delta_{23},-\Delta_{13},\Delta_{12})
\]
for some non-zero-divisor $a\in R$. Reducing to the geometric fiber
through $x$, the ideal generated by the $\overline f_i$ is either the
unit ideal or has height two. Hence $\overline a$ cannot be a
non-unit, and therefore $a$ is a unit in $R$.

But $\delta$ is represented in these trivializations by
\[
(\Delta_{23},-\Delta_{13},\Delta_{12}),
\]
whereas $\lambda$ is represented by \((f_1,f_2,f_3)\). Thus multiplication by $a$ gives a local isomorphism
\[
\theta:
\det(\mathcal K)^\vee
\xlongrightarrow{\sim}
\widetilde{\mathcal L}
\]
satisfying \(\theta\circ\delta=\lambda\).

On the complement of $\mathcal Z$ this is precisely the canonical
determinant isomorphism associated with the exact sequence of vector
bundles
\[
0\lar\mathcal K
\lar\mathcal O^{\oplus3}
\lar\widetilde{\mathcal L}
\lar0.
\]
Since $\mathcal Z$ has finite fibers over $T$, its complement is
schematically dense in $\PP^2_T$ with codimension two. Hence the local isomorphisms agree
on overlaps and glue uniquely. This proves~(b).

Finally, formation of $\mathcal K$ commutes with arbitrary base
change, and the same is true for exterior powers, determinants, and
the canonical wedge morphism. The isomorphism $\theta$ is uniquely
characterized by
\[
\theta\circ\delta=\lambda,
\]
so it too is compatible with base change. This proves~(c).
\end{proof}

The following important flatness lemma is included here for the sake of completeness.

\begin{Lemma}[Fiberwise injectivity and flat cokernel]
\label{lem:flatness-finite-cokernel}
Let $q:X\lar T$ be a morphism of locally Noetherian schemes, let
$\mathcal F$ be a coherent sheaf on $X$, let $\mathcal L$ be a
coherent sheaf which is $T$-flat, and let
\[
\varphi:\mathcal F\longrightarrow\mathcal L
\]
be a morphism. Assume that, for every geometric point $t\lar T$, the
induced morphism
\[
\varphi_t:\mathcal F_t\longrightarrow\mathcal L_t
\]
is injective. Then $\varphi$ is injective and \(\mathcal C\doteq\coker(\varphi)\)
is $T$-flat.

If, moreover, $q$ is projective, $\mathcal L$ is invertible, and
$\mathcal C_t$ has length $j$ for every geometric point $t\lar T$,
then
\[
\mathcal I_{\mathcal Z}
\doteq
\mathcal F\otimes\mathcal L^\vee
\hookrightarrow\mathcal O_X
\]
defines a finite flat family
\(\mathcal Z\subset X\) of closed subschemes of degree $j$ over $T$, with
\(\mathcal O_{\mathcal Z}
\simeq
\mathcal C\otimes\mathcal L^\vee
\).
\end{Lemma}

\begin{proof}
The assertion is local on $X$. Let $x\in X$, $t=q(x)$, and set \(A=\mathcal O_{T,t}\). By hypothesis,
\[
\varphi_x\otimes_A\kappa(t)
\]
is injective, while $\mathcal L_x$ is flat over $A$. The local
criterion for flatness
\cite[Section~20, Application~20.2]{matsumuraCA}
therefore implies that $\varphi_x$ is injective and that its cokernel
is flat over $A$. Since this holds for every $x$, $\varphi$ is
injective and \(
\mathcal C=\coker(\varphi)\) is flat over $T$.

Assume now that $q$ is projective, $\mathcal L$ is invertible, and
every geometric fiber $\mathcal C_t$ has length $j$. Then
$\Supp(\mathcal C)\lar T$ is proper with finite fibers, hence finite.
Tensoring
\[
0\lar\mathcal F\lar\mathcal L\lar\mathcal C\lar0
\]
by $\mathcal L^\vee$ gives
\[
0\lar
\mathcal F\otimes\mathcal L^\vee
\lar
\mathcal O_X
\lar
\mathcal C\otimes\mathcal L^\vee
\lar0.
\]
Thus \(\mathcal I_{\mathcal Z}
=
\mathcal F\otimes\mathcal L^\vee\) is an ideal sheaf and
\(\mathcal O_{\mathcal Z}
\simeq
\mathcal C\otimes\mathcal L^\vee\) is flat over $T$ with fibers of length $j$. Hence
$\mathcal Z\lar T$ is finite flat of degree $j$.
\end{proof}

In the next theorem, we show $Z_j$ is isomorphic to an open subset of a Quot scheme, enabling us to describe both the tangent space and an obstruction space locally at each point.

\begin{Theorem}\label{thm:open-Z-j-quot-scheme}
Fix $d \geq 1$ and $j \geq 0$. The functor $\Trip_{d,j}$ is a sub-functor of Grothendieck's Quot functor $\QuotF_{\mathbb{P}^2}^{P_j}(\mathcal{O}_{\mathbb{P}^2}^{\oplus 3})$ of quotients of $\mathcal{O}_{\mathbb{P}^2}^{\oplus 3}$ over $\mathbb{P}^2$ with fixed Hilbert polynomial $P_j = \binom{t+d+2}{2} - j$. Moreover, the induced morphism of schemes
$$
i : Z_j \lar \Quot_{\mathbb{P}^2}^{P_j}(\mathcal{O}_{\mathbb{P}^2}^{\oplus 3})
$$
is an injective open immersion, with
$$
\text{im }i = \{[\rho : \mathcal{O}_{\mathbb{P}^2}^{\oplus 3} \twoheadrightarrow \mathcal{F}] : \mathcal{F} \text{ is torsion-free and }h^0(\ker(\rho)) = 0\}.
$$
In particular, the tangent space of $Z_j$ at a point $p = [f_1, f_2, f_3] = J$, with $Z = V(J)$, is isomorphic to $\Hom(\T_J, \mathcal{I}_Z(d))$, and $\Ext^1(\T_J, \mathcal{I}_Z(d))$ is an associated obstruction space.
\end{Theorem}

\begin{proof}
Let $T$ be a $k$-scheme and let \(
[\mathcal L,\mathbf f]\in\Trip_{d,j}(T)\). The associated evaluation morphism has image
\[
\mathcal I_{\mathcal Z}
\subseteq
\mathcal O_{\PP^2_T},
\]
where $\mathcal Z\lar T$ is finite flat of degree $j$. After twisting
by
\[
q^*\mathcal L\otimes p^*\mathcal O_{\PP^2}(d),
\]
we obtain a $T$-flat quotient
\[
\mathcal O_{\PP^2_T}^{\oplus3}
\twoheadrightarrow
\mathcal I_{\mathcal Z}
\otimes q^*\mathcal L
\otimes p^*\mathcal O_{\PP^2}(d)
\]
with Hilbert polynomial
\[
P_j(t)=\binom{t+d+2}{2}-j.
\]
This construction is invariant under the equivalence defining
$\Trip_{d,j}$ and commutes with arbitrary base change. It therefore
defines a natural transformation
\[
\gamma_j:
\Trip_{d,j}
\longrightarrow
\QuotF_{\PP^2}^{P_j}
(\mathcal O_{\PP^2}^{\oplus3}).
\]

Let $U_{d,j}$ denote the subfunctor of the Quot functor consisting of
quotients
\[
\rho:
\mathcal O_{\PP^2_T}^{\oplus3}
\twoheadrightarrow\mathcal F
\]
such that every geometric fiber $\mathcal F_t$ is torsion-free and \(H^0(\PP^2_t,\ker\rho_t)=0\). Both conditions are open in flat families, the second by upper
semicontinuity. Hence $U_{d,j}$ is represented by an open subscheme
of the Quot scheme, and $\gamma_j$ factors through $U_{d,j}$.

We construct the inverse transformation \(\delta_j:U_{d,j}\longrightarrow\Trip_{d,j}\). Let
\[
\rho:
\mathcal O_{\PP^2_T}^{\oplus3}
\twoheadrightarrow\mathcal F
\]
be a $T$-point of $U_{d,j}$ and put \(\mathcal K\doteq\ker(\rho)\). Since $\mathcal F$ is $T$-flat, so is $\mathcal K$, and formation
of $\mathcal K$ commutes with base change. For every geometric
$t\lar T$ we have the associated short exact sequence
\[
0\longrightarrow\mathcal K_t
\longrightarrow\mathcal O_{\PP^2_t}^{\oplus3}
\longrightarrow\mathcal F_t
\longrightarrow0.
\]
Since $\mathcal F_t$ is torsion-free of rank one on the regular
surface $\PP^2_t$, the depth lemma shows that $\mathcal K_t$ is
locally free of rank two. Hence $\mathcal K$ is locally free of rank
two on $\PP^2_T$.

Set \(
\mathcal L'\doteq\det(\mathcal K)^\vee\). The wedge construction, as in Lemma~\ref{lem:compatibility-determinant} associated with the inclusion
\[
\mathcal K\hookrightarrow\mathcal O_{\PP^2_T}^{\oplus3}
\]
gives a canonical morphism \(
\mathcal O_{\PP^2_T}^{\oplus3}
\longrightarrow\mathcal L' \) which annihilates $\mathcal K$, and therefore factors as
\( \varphi:
\mathcal F\longrightarrow\mathcal L'\).

For a geometric point $t\lar T$, let
$V_t\subset\PP^2_t$ be the locus where $\mathcal F_t$ is locally
free. On $V_t$ the sequence
\[
0\longrightarrow\mathcal K_t|_{V_t}
\longrightarrow\mathcal O_{V_t}^{\oplus3}
\longrightarrow\mathcal F_t|_{V_t}
\longrightarrow0
\]
is an exact sequence of vector bundles, and taking determinants
shows that
\[
\varphi_t|_{V_t}:
\mathcal F_t|_{V_t}
\xlongrightarrow{\sim}
\det(\mathcal K_t)^\vee|_{V_t}.
\]
Hence $\ker(\varphi_t)$ has rank zero. Since
$\mathcal F_t$ is torsion-free, \(
\ker(\varphi_t)=0\), and thus $\varphi_t$ is injective for every geometric $t$.

Since $\varphi_t$ is injective for every geometric point $t\lar T$,
Lemma~\ref{lem:flatness-finite-cokernel} implies that
\[
\varphi:\mathcal F\longrightarrow\mathcal L'
\]
is injective and that \(\mathcal C\doteq\coker(\varphi)\) is $T$-flat.

For every geometric point $t\lar T$, the fixed Hilbert polynomial of
$\mathcal F_t$ gives
\[
\det(\mathcal K_t)^\vee\simeq\mathcal O_{\PP^2_t}(d).
\]
Indeed, from
\[
0\lar\mathcal K_t\lar\mathcal O_{\PP^2_t}^{\oplus3}
\lar\mathcal F_t\lar0
\]
we have
\[
\det(\mathcal K_t)^\vee\simeq\det(\mathcal F_t),
\]
while the Hilbert polynomial
\[
P_{\mathcal F_t}(m)
=
\binom{m+d+2}{2}-j
\]
shows that $\mathcal F_t$ has rank one and first Chern class $d$.
Thus
\[
\det(\mathcal F_t)\simeq\mathcal O_{\PP^2_t}(d).
\]

Consequently the exact sequence
\[
0\longrightarrow\mathcal F_t
\longrightarrow\mathcal O_{\PP^2_t}(d)
\longrightarrow\mathcal C_t
\longrightarrow0
\]
gives
\[
P_{\mathcal C_t}(m)=j.
\]
Hence $\mathcal C_t$ is zero-dimensional of length $j$.

Since $\mathcal L'$ is invertible and every geometric fiber
$\mathcal C_t$ has length $j$, the second part of
Lemma~\ref{lem:flatness-finite-cokernel} shows that
\[
\mathcal I_{\mathcal Z}
\doteq
\mathcal F\otimes(\mathcal L')^\vee
\hookrightarrow
\mathcal O_{\PP^2_T}
\]
defines a finite flat family
\(
\mathcal Z\subset\PP^2_T
\) of degree $j$, with \(
\mathcal O_{\mathcal Z}
\simeq
\mathcal C\otimes(\mathcal L')^\vee.
\)

Moreover, the line bundle
\[
\mathcal N
\doteq
\mathcal L'\otimes p^*\mathcal O_{\PP^2}(-d)
\]
restricts to $\mathcal O_{\PP^2}$ on every geometric fiber of $q$.
Since \(H^0(\PP^2,\mathcal O_{\PP^2})=k\) and \(H^1(\PP^2,\mathcal O_{\PP^2})=0\), cohomology and base change
\cite[Chapter~III, Theorem~12.11]{hartshorne2013algebraic}
shows that
\[
\mathcal M\doteq q_*\mathcal N
\]
is a line bundle on $T$ and that its formation commutes with base
change. The adjunction morphism
\[
q^*\mathcal M\longrightarrow\mathcal N
\]
restricts on every geometric fiber to the evaluation isomorphism
\[
H^0(\PP^2,\mathcal O_{\PP^2})\otimes
\mathcal O_{\PP^2}
\xlongrightarrow{\sim}
\mathcal O_{\PP^2}.
\]
Hence it is an isomorphism, and therefore \(\mathcal L'
\simeq q^*\mathcal M\otimes p^*\mathcal O_{\PP^2}(d)\).

The composition
\[
\mathcal O_{\PP^2_T}^{\oplus3}
\xlongrightarrow{\rho}
\mathcal F
\xlongrightarrow{\varphi}
\mathcal L'
\]
therefore defines a triple of degree-$d$ forms with coefficients in
$\mathcal M$. The condition 
\[
H^0(\PP^2_t,\mathcal K_t)=0
\] 
says that the three forms are linearly independent on every geometric fiber, and the fact that their base scheme $\mathcal Z_t$
is zero-dimensional implies that they have no common divisorial
factor. Thus we obtain an element
\(
\delta_j(T)(\rho)\in\Trip_{d,j}(T).
\)
All the constructions above commute with arbitrary base change, so
$\delta_j$ is a natural transformation.

By Lemma~\ref{lem:compatibility-determinant}, if one starts with
$[\mathcal L,\mathbf f]\in\Trip_{d,j}(T)$, applies $\gamma_j$, and
then performs the determinant reconstruction, one recovers the
original triple up to the equivalence defining $\Trip_{d,j}$.
Hence \(
\delta_j\circ\gamma_j=\id\).
Conversely, starting with $\rho\in U_{d,j}(T)$, the reconstructed
ideal sheaf satisfies
\[
\mathcal I_{\mathcal Z}\otimes\mathcal L'
=
\mathcal F,
\]
and the reconstructed evaluation morphism is
$\varphi\circ\rho$. Hence applying $\gamma_j$ recovers the original
quotient $\rho$, so \(
\gamma_j\circ\delta_j=\id.
\) Therefore, we describe a natural isomorphism \(
\Trip_{d,j}\simeq U_{d,j}
\) as functors. By Yoneda, \(
Z_j\simeq U_{d,j}\), and consequently
\[
i:Z_j\hookrightarrow
\Quot_{\PP^2}^{P_j}
(\mathcal O_{\PP^2}^{\oplus3})
\]
is an open immersion with the stated image.

Finally, for a point
\[
p=[f_1,f_2,f_3]\in Z_j,
\]
the corresponding quotient is
\[
0\longrightarrow\mathcal T_J
\longrightarrow\mathcal O_{\PP^2}^{\oplus3}
\longrightarrow\mathcal I_Z(d)
\longrightarrow0.
\]
The standard deformation theory of the Quot scheme (see, for example, \cite[Corollary 4.4.6]{sernesi2006deformations}) gives
\[
T_pZ_j
\simeq
\Hom(\mathcal T_J,\mathcal I_Z(d)),
\]
while \(
\Ext^1(\mathcal T_J,\mathcal I_Z(d))
\) is an obstruction space.
\end{proof}

Next, we review an important characterization from the deformation theory of sheaves.

\begin{Lemma}
Fix $d \geq 1$ and $j \geq 0$. For every flat first-order deformation $J_\epsilon$ of a particular point $[(f_1, f_2, f_3)] = [J]$ of $Z_j$, let $Z_{\epsilon} = V(J_\epsilon) \subset \mathbb{P}^2_{\epsilon} = \mathbb{P}^2 \times_k \mathbb{D}$. Then, there are induced extensions
$$
[\mathcal{I}_{Z_\epsilon}] \in \Ext^1(\mathcal{I}_Z, \mathcal{I}_Z) \text{ and }[\T_\epsilon] \in \Ext^1(\T_J, \T_J)
$$
where $Z_{\epsilon} = V(J_\epsilon)$, in the sense of flat first-order deformations of coherent sheaves of $\mathbb{P}^2$ over $k$.
\end{Lemma}

\begin{proof}
Note that a first-order deformation $J_\epsilon = (F_1, F_2, F_3)$ induces a first-order deformation of the ideal sheaf $\mathcal{I}_Z(d)$ by $\mathcal{I}_{Z_{\epsilon}}$. Moreover, the associated sheaf $\T_{\epsilon} = \ker( \mathcal{O}_{\mathbb{P}^2_{\epsilon}}^{\oplus 3} \twoheadrightarrow \mathcal{I}_{Z_\epsilon}(d) )$ is also a flat first-order deformation of $\T_J$, since it is flat over $\mathbb{D}$ and we recover $\T_J$ at the special fiber. The fact that these two Ext groups classify first-order deformations of coherent sheaves is classical (see, for example, \cite[Proposition 2.6]{hartshorne2010deformation}).
\end{proof}

Using the description in Theorem~\ref{thm:open-Z-j-quot-scheme},
$$
T_{[J]} Z_j \simeq \Hom(\T_J, \mathcal{I}_Z(d))
$$
for a triple $[J] \in Z_j$ with $Z = V(J)$, the association above for infinitesimal deformations can be presented as follows. After applying the functor $\Hom(-, \mathcal{I}_Z(d))$ to the short exact sequence
$$
0 \lar \T_J \lar \mathcal{O}_{\mathbb{P}^2}^{\oplus 3} \lar \mathcal{I}_Z(d) \lar 0,
$$
the induced connecting morphism
$$
\delta : \Hom(\T_J, \mathcal{I}_Z(d)) \lar \Ext^1(\mathcal{I}_Z(d), \mathcal{I}_Z(d)) \simeq \Ext^1(\mathcal{I}_Z, \mathcal{I}_Z)
$$
realizes the association of each first order infinitesimal deformation $J_\epsilon$ to the associated deformation of ideal sheaves $\mathcal{I}_{Z_\epsilon}(d)$ with $Z_{\epsilon} = V(J_\epsilon) \subset \mathbb{P}^2_{\epsilon}$. This can be seen by following the associated push-out construction.

When $[J] \in Z_j$ is a point with $\Ext^1(\T_J, \mathcal{I}_Z(d)) = 0$, then $[J]$ is a smooth point. When this occurs, we may compute the dimension of $Z_j$ at $[J]$ using the Hirzebruch-Riemann-Roch (HRR) Theorem, as follows:

\begin{Lemma}\label{lem:dimension-Z-j-smooth-points}
Let $d \geq 1$, $j \geq 0$ and let $[J] \in Z_j$ be a point such that $\Ext^1(\T_J, \mathcal{I}_Z(d))$ vanishes, where $Z = V(J)$. Then $Z_j$ is smooth at $[J]$, of dimension
$$
\dim_{[J]} Z_j = \frac{3}{2}d^2 + \frac{9}{2}d + 2 - j.
$$
In particular $\codim_{\mathbb{P}(W_d)} Z_j = j$ at $[J]$.
\end{Lemma}

\begin{proof}
Applying the functor $\Hom(\T_J, -)$ to the short exact sequence $
0 \lar \T_J \lar \mathcal{O}_{\mathbb{P}^2}^{\oplus 3} \lar \mathcal{I}_Z(d) \lar 0$, we get the long exact sequence
\begin{align}\label{eq:les-1-HRR-unobstructed}
0 &\lar \Hom(\T_J, \T_J) \lar H^0(\T_J^\vee)^3 \lar \Hom(\T_J, \mathcal{I}_Z(d)) \lar\\
&\lar \Ext^1(\T_J, \T_J) \lar H^1(\T_J^\vee)^3 \lar \Ext^1(\T_J, \mathcal{I}_Z(d)) = 0
\end{align}
and a short exact sequence 
\begin{equation}\label{eq:les-2-HRR-unobstructed}
0 \lar \Ext^2(\T_J, \T_J) \lar H^2(\T_J^\vee)^3 \lar \Ext^2(\T_J, \mathcal{I}_Z(d)) \lar 0.    
\end{equation}
From Serre duality, we get
$H^2(\T_J^\vee) \simeq H^0(\T_J(-3))^* = 0$ and thus both end-points of \eqref{eq:les-2-HRR-unobstructed} are zero. From additivity of dimensions at \eqref{eq:les-1-HRR-unobstructed}, we conclude
\begin{equation}\label{eq:additivity-Chernclasses-unobstructed}
\dim(\Hom(\T_J, \mathcal{I}_Z(d))) = 3\mathcal{X}(\T_J^\vee) - \mathcal{X}(\T_J \otimes \T_J^\vee)    
\end{equation}
since $\Ext^i(\T_J, \T_J) \simeq \Ext^i(\mathcal{O}_{\mathbb{P}^2}, \T_J \otimes \T_J^\vee) \simeq H^i(\T_J \otimes \T_J^\vee)$ for $i \geq 0$ (see, for example, \cite[Chapter III, Proposition 6.7]{hartshorne2013algebraic}). We now compute both Euler characteristics using the Hirzebruch-Riemann-Roch formula (see, for example, \cite[Corollary 15.2.1]{fulton2013intersection}). Since $\ch(\T_J) = (2, -d, j - d^2/2)$, we conclude $\ch(\T_J^\vee) = (2, d, j - d^2/2)$ and thus
\begin{align*}
\ch(\T_J \otimes \T_J^\vee) = \ch(\T_J) \cdot \ch(T_J^\vee) &= 4 + \left(4j - 3d^2 \right) H^2\\
\mathcal{X}(\T_J \otimes \T_J^\vee) = \int \ch(\T_J \otimes \T_J^\vee) \cdot \text{Todd}(\mathbb{P}^2) &= \int \left(4 + (4j - 3d^2)H^2\right) \cdot \left( 1 + \frac{3}{2}H + H^2 \right)\\
&= 4 + 4j - 3d^2.
\end{align*}
On the other side,
\begin{align*}
\mathcal{X}(\T_J^\vee) = \int \ch(\T_J^\vee) \cdot \text{Todd}(\mathbb{P}^2) &= \int \left( 2 + d H + \left(j - \frac{d^2}{2}\right)H^2 \right) \cdot \left( 1 + \frac{3}{2}H + H^2 \right)\\
&= 2 + \frac{3}{2}d + j - \frac{d^2}{2}.
\end{align*}
Putting these together with \eqref{eq:additivity-Chernclasses-unobstructed} and Theorem~\ref{thm:open-Z-j-quot-scheme}, we conclude
$$
\dim T_{[J]} Z_j = \dim \Hom(\T_J, \mathcal{I}_Z(d)) = \frac{3}{2}d^2 + \frac{9}{2}d + 2 - j.
$$
\end{proof}

We emphasize that the vanishing of the Quot obstruction group
\[
\Ext^1(\T_J,\mathcal I_Z(d))
\]
is sufficient, but not necessary, for smoothness of $Z_j$ at $[J]$.
Indeed, there are free points for which this group is nonzero although
$Z_j$ is smooth; see Corollary~\ref{cor:free-points-picture}.

We next compute the tangent dimension and the Quot obstruction group
at free points.

\begin{Proposition}\label{prop:dim-at-free-points}
Let $d \geq 1$, $j \geq 1$, and let $[f] = [J] \in Z_j$ be a degree-$d$ triple. If $\syz(J) \simeq R(-e)\oplus R(e-d)$, then 
\begin{itemize}
    \item[(a)] The dimension of the tangent space of $Z_j$ at $[f]$ is
$$
\dim(T_{[f]} Z_j) = \binom{d+e+2}{2} + \binom{2d-e+2}{2}- 2j + h^0(\mathcal{O}_{\mathbb{P}^2}(d-2e-3)).
$$
Alternatively,
\begin{align*}
\dim(T_{[f]} Z_j) &= \frac{1}{2}d^2 + de -e^2 + \frac{9}{2}d + 2 \hspace{3cm} \text{ when }2e\geq d-2,\\
\dim(T_{[f]} Z_j) &= \frac{1}{2}d^2 + de -e^2 + \frac{9}{2}d + 2 + \binom{d-2e-1}{2} \ \text{ when }2e\leq d-3.
\end{align*}
    \item[(b)] Whenever $2e \geq d-2$, the scheme $Z_j$ is smooth at $[f]$.
\end{itemize}
\end{Proposition}

\begin{proof}
If $J$ is free, then $\T_J \simeq \mathcal{O}_{\mathbb{P}^2}(-e) \oplus \mathcal{O}_{\mathbb{P}^2}(e-d)$ with $1 \leq e \leq \lfloor d/2 \rfloor$ the initial degree. By Theorem~\ref{thm:open-Z-j-quot-scheme}, there is an isomorphism
$$
T_{[f]} Z_j \simeq \Hom(\T_J, \mathcal{I}_Z(d)) \simeq H^0(\mathcal{I}_Z(d+e)) \oplus H^0(\mathcal{I}_Z(2d-e)).
$$
From the short exact sequence
$$
0 \lar \mathcal{I}_Z(m) \lar \mathcal{O}_{\mathbb{P}^2}(m) \lar \mathcal{O}_Z \lar 0
$$
and using $h^1(\mathcal{O}_{\mathbb{P}^2}(m)) = 0$, we obtain
$$
h^0(\mathcal{I}_Z(m)) = \binom{m+2}{2}-j+h^1(\mathcal{I}_Z(m))
$$
for $m \geq 0$ and $j = \deg(R/J)$. Denoting by $\mathcal{F}_0 = \mathcal{O}_{\mathbb{P}^2}^{\oplus 3}(-d)$ and $\mathcal{F}_1 = \mathcal{O}_{\mathbb{P}^2}(-e-d) \oplus \mathcal{O}_{\mathbb{P}^2}(-2d+e)$, the Hilbert-Burch resolution for $\mathcal{I}_Z$ is of the form:
$$
0 \lar \mathcal{F}_1 \xlongrightarrow{\psi} \mathcal{F}_0 \lar \mathcal{I}_Z  \lar 0.
$$
Twisting by $\mathcal{O}_{\mathbb{P}^2}(m)$ and taking cohomology yields a long exact sequence
$$
0 \lar H^1(\mathcal{I}_Z(m)) \lar H^2(\mathcal{F}_1(m)) \xlongrightarrow{\alpha} H^2(\mathcal{F}_0(m)) \lar H^2(\mathcal{I}_Z(m)) = 0
$$
giving $
h^1(\mathcal{I}_Z(m)) = \dim \ker\left( H^2(\mathcal{F}_1(m)) \xlongrightarrow{\alpha} H^2(\mathcal{F}_0(m)) \right)$. Evaluating first on $m = d+e$, we obtain:
\[
\mathcal F_1(d+e)
=
\mathcal O_{\PP^2}
\oplus
\mathcal O_{\PP^2}(2e-d),
\qquad
\mathcal F_0(d+e)
=
\mathcal O_{\PP^2}(e)^{\oplus3}.
\]
Since $e \geq 1$, $H^2(\mathcal{F}_0(d+e)) = 0$, so $\alpha = 0$. Moreover, $H^2(\mathcal{O}_{\mathbb{P}^2}) = 0$, so
$$
h^1(\mathcal{I}_Z(d+e)) = h^2(\mathcal{O}_{\mathbb{P}^2}(2e-d)) = h^0(\mathcal{O}_{\mathbb{P}^2}(d-2e-3))
$$
by Serre duality. On the other hand, evaluating for $m = 2d-e$, we obtain
$$
\mathcal{F}_1(2d-e) = \mathcal{O}_{\mathbb{P}^2}(d-2e) \oplus \mathcal{O}_{\mathbb{P}^2},
$$
and since $e \leq \lfloor d/2\rfloor$, then $H^2(\mathcal{O}_{\mathbb{P}^2}) = H^2(\mathcal{O}_{\mathbb{P}^2}(d-2e)) = 0$ and thus $h^1(\mathcal{I}_Z(2d-e)) = 0$. From this, $(a)$ follows.

From the splitting of $\T_J$, it follows that $\Ext^1(\T_J, \mathcal{O}_{\mathbb{P}^2}^{\oplus 3}) = 0$ and
$$
\Ext^2(\T_J, \mathcal{O}_{\mathbb{P}^2}^{\oplus 3}) \simeq H^2(\T_J^\vee)^3 \simeq H^2(\mathcal{O}_{\mathbb{P}^2}(e))^3 \oplus H^2(\mathcal{O}_{\mathbb{P}^2}(d-e))^3 = 0,
$$
since $e \geq 1$ and $d-e \geq \lceil d/2 \rceil \geq 0$. Then, applying the functor $\Hom(\T_J, -)$ to the short exact sequence
$$
0 \lar \T_J \lar \mathcal{O}_{\mathbb{P}^2}^{\oplus 3} \lar \mathcal{I}_Z(d) \lar 0,
$$
we obtain an isomorphism
\[
\Ext^1(\T_J,\mathcal I_Z(d))
\simeq
\Ext^2(\T_J,\T_J)
\simeq
H^2(\T_J\otimes\T_J^\vee).
\]
Thus the obstruction group furnished by the Quot deformation theory
is
\[
\Ext^1(\T_J,\mathcal I_Z(d))
\simeq
H^0\bigl(\mathcal O_{\PP^2}(d-2e-3)\bigr)^\vee.
\]
In particular, it vanishes whenever $2e\geq d-2$, and~(b) follows.
\end{proof}

\section{Mapping triples to associated closed subsets}\label{sec:hilbert-map}

We now study the natural morphism
\[
\Phi_j:Z_j\longrightarrow\Hilbert^j(\PP^2),
\qquad
[J]\longmapsto V(J).
\]
By Fogarty's theorem
\cite[Proposition~2.3 and Theorem~2.4]{Fogarty1968},
$\Hilbert^j(\PP^2)$ is smooth and irreducible of dimension $2j$.
After constructing $\Phi_j$, we describe its differential and derive
smoothness and irreducibility criteria for the strata $Z_j$.

\begin{Theorem}\label{thm:hilb-map-j}
For every $d\geq1$ and $j\geq0$, there is a morphism
\[
\begin{aligned}
\Phi_j:Z_j&\longrightarrow\Hilbert^j(\PP^2),\\
[f_1:f_2:f_3]
&\longmapsto
V(f_1,f_2,f_3).
\end{aligned}
\]
Moreover, for every $[Z]\in\Hilbert^j(\PP^2)$, the fiber
$\Phi_j^{-1}([Z])$ is, whenever nonempty, an open subset of
\[
\PP\bigl(H^0(\mathcal I_Z(d))^{\oplus3}\bigr).
\]
\end{Theorem}

\begin{proof}
Let $\HilbertF^j_{\mathbb{P}^2_k}$ denote the usual Hilbert functor of flat families of closed subschemes of $\mathbb{P}^2$ with constant Hilbert polynomial $j$. We define the natural transformation, over each $T \in \Sch_k$:
\begin{align*}
\eta_j(T) : \Trip_{d,j}(T) &\lar \HilbertF^j_{\mathbb{P}^2/k}(T)\\
[\mathcal{L}, \textbf{f}] &\longmapsto \mathcal{Z}_{[\mathcal{L}, \textbf{f}]} \subset \mathbb{P}^2_T.
\end{align*}
Since the pullback by morphisms $h: S \lar T$ corresponds to the base-change of $\mathcal{Z}$, as shown in Lemma~\ref{lem:functors-of-triples}, it follows that $\eta_j$ is a natural transformation, and since 
$$
\Trip_{d,j} \simeq \Hom(-, Z_j) \text{ and } \HilbertF^j_{\mathbb{P}^2_k} \simeq \Hom(-, \Hilbert^j(\mathbb{P}^2)),
$$
the existence of the morphism $\Phi_j$ follows from Yoneda's Lemma.

Let $[Z] \in \Hilbert^j(\mathbb{P}^2)$ with $\Phi_j([f]) = Z$. Then
$$
\Phi_j^{-1}([Z]) = \{[g]=[g_1, g_2, g_3] \in \mathbb{P}(H^0(\mathcal{I}_Z(d))^{3}) : \dim_k \{g_1, g_2, g_3\} = 3\text{ and }\sigma_{[g]} \text{ is surjective}\},
$$
where $\sigma_{[g]} = (g_1, g_2, g_3): \mathcal{O}_{\mathbb{P}^2}^{\oplus 3} \lar \mathcal{I}_Z(d)$ is the associated morphism to the triple. We claim both conditions are open inside the space $\mathbb{P}(H^0(\mathcal{I}_Z(d))^3)$: the first is given by $g_1 \wedge g_2 \wedge g_3 \neq 0$, so it is clearly the complement of a Zariski closed subset. 

It remains to show that the condition that
\[
\sigma_{[g]}:
\mathcal O_{\PP^2}^{\oplus3}
\longrightarrow
\mathcal I_Z(d)
\]
be surjective is open. To see this, we set \(P_Z\doteq
\PP\bigl(H^0(\mathcal I_Z(d))^{\oplus3}\bigr)\)
and let \(p:\PP^2\times P_Z\lar\PP^2\) and \(q:\PP^2\times P_Z\lar P_Z\) be the projections. The universal triple determines a morphism
\[
\Sigma:
\mathcal O_{\PP^2\times P_Z}^{\oplus3}
\otimes q^*\mathcal O_{P_Z}(-1)
\longrightarrow
p^*\mathcal I_Z(d).
\]
Let \(
\mathcal C\doteq\coker(\Sigma)\). The support \(
\Supp(\mathcal C)
\subseteq
\PP^2\times P_Z\) is closed. A point $[g]\in P_Z$ corresponds to a surjective morphism
$\sigma_{[g]}$ if and only if
\[
\Supp(\mathcal C)\cap
\bigl(\PP^2\times\{[g]\}\bigr)
=
\varnothing.
\]
Since $q$ is proper, the image \(q(\Supp(\mathcal C))\) is closed. Therefore the locus \(P_Z\setminus q(\Supp(\mathcal C))\) on which $\sigma_{[g]}$ is surjective is open. Intersecting this open subset with the open locus where
\[
g_1\wedge g_2\wedge g_3\neq0
\]
gives precisely $\Phi_j^{-1}([Z])$. Hence every fiber of $\Phi_j$ is
an open subset of \(\PP\bigl(H^0(\mathcal I_Z(d))^{\oplus3}\bigr)\).
\end{proof}

In the present work, we use the morphism
\[
\Phi_j:Z_j\longrightarrow\Hilbert^j(\PP^2)
\]
together with the irreducibility of its fibers and of
$\Hilbert^j(\PP^2)$ to deduce irreducibility of several strata
$Z_j$. We record the argument here for later use.

\begin{Lemma}\label{lem:irreducibility-from-Hilbert-map}
Assume that $Z_j$ is nonempty and that
\[
\Phi_j:Z_j\longrightarrow\Hilbert^j(\PP^2)
\]
is smooth. Then $Z_j$ is irreducible and
$\Phi_j(Z_j)$ is a dense open subset of
$\Hilbert^j(\PP^2)$.
\end{Lemma}

\begin{proof}
Since $\Phi_j$ is smooth, it is open. Hence \(B_j\doteq\Phi_j(Z_j)
\) is a nonempty open subset of the irreducible scheme
$\Hilbert^j(\PP^2)$, and therefore is irreducible.

By Theorem~\ref{thm:hilb-map-j}, every nonempty fiber of $\Phi_j$
is an open subset of a projective space, and hence is irreducible.
Now let $U,V\subseteq Z_j$ be nonempty open subsets. Since $\Phi_j$
is open, $\Phi_j(U)$ and $\Phi_j(V)$ are nonempty open subsets of
$B_j$, so they intersect. Choose
\[
[Z]\in\Phi_j(U)\cap\Phi_j(V).
\]
Then
\[
U\cap\Phi_j^{-1}([Z])
\qquad\text{and}\qquad
V\cap\Phi_j^{-1}([Z])
\]
are nonempty open subsets of the irreducible fiber
$\Phi_j^{-1}([Z])$, hence intersect. Thus $U\cap V\neq\emptyset$,
and $Z_j$ is irreducible.

Finally, $B_j$ is nonempty and open in the irreducible scheme
$\Hilbert^j(\PP^2)$, so it is dense.
\end{proof}

Theorem~\ref{thm:open-Z-j-quot-scheme} enables a description of the differential map to the Hilbert scheme:

\begin{Proposition}\label{prop:differential-Hilb-j}
Let $d \geq 1$ and $j \geq 1$, and let $\Phi_j : Z_j \lar \Hilbert^j(\mathbb{P}^2)$ be the morphism defined in Theorem~\ref{thm:hilb-map-j}. Let us fix a point $p = [f_1, f_2, f_3] \in Z_j$ generating the ideal $J = (f_1, f_2, f_3)$ and let $\Phi_j(p) = [Z]$ as a subscheme on $\mathbb{P}^2$. Then, the derivative of $\Phi_j$ at $p$
$$
d_{p}\Phi_j : T_p Z_j \simeq \Hom(\T_J, \mathcal{I}_Z(d)) \lar \Hom(\mathcal{I}_Z, \mathcal{O}_Z) \simeq T_{\Phi_j(p)}\Hilbert^j(\mathbb{P}^2)
$$
may be described via a commutative diagram 
\begin{center}
\begin{tikzcd}
{\Hom(\T_J,\mathcal I_Z(d))}
    \arrow[rr,"\delta"]
    \arrow[d,"d_p\Phi_j"']
&&
{\Ext^1(\mathcal I_Z(d),\mathcal I_Z(d))}
    \arrow[d,"\tau","\simeq"']\\
{\Hom(\mathcal I_Z,\mathcal O_Z)}
    \arrow[rr,"\partial",hook]
&&
{\Ext^1(\mathcal I_Z,\mathcal I_Z)}
\end{tikzcd}
\end{center}
where $\delta$ is the connecting morphism obtained after applying $\Hom(-, \mathcal{I}_Z(d))$ to the short exact sequence
$$
0 \lar \T_J \lar \mathcal{O}_{\mathbb{P}^2}^{\oplus 3} \lar \mathcal{I}_Z(d) \lar 0,
$$
and $\partial$ is the connecting morphism obtained by applying $\Hom(\mathcal{I}_Z,-)$ to 
$$
0 \lar \mathcal{I}_Z \lar \mathcal{O}_{\mathbb{P}^2} \lar \mathcal{O}_Z \lar 0.
$$
In particular, one concludes that
\[
\ker(d_p\Phi_j)
\simeq
\frac{H^0(\mathcal I_Z(d))^{\oplus3}}
{k\cdot(f_1,f_2,f_3)}
\]
and there is a natural injection \(
\coker(d_p\Phi_j)
\hookrightarrow
H^1(\mathcal I_Z(d))^{\oplus3}\). Moreover
\[
\dim\ker(d_p\Phi_j)
=
3h^0(\mathcal I_Z(d))-1,
\]
and, if \(
H^1(\mathcal I_Z(d))=0\), then $d_p\Phi_j$ is surjective.
\end{Proposition}

\begin{proof}
The tangent space at $[Z] = [V(J)] \in \Hilbert^j(\mathbb{P}^2)$ is isomorphic to $\Hom(\mathcal{I}_Z, \mathcal{O}_Z)$. Applying $\Hom(\mathcal{I}_Z, -)$ to
$$
0 \lar \mathcal{I}_Z \xlongrightarrow{i} \mathcal{O}_{\mathbb{P}^2} \lar \mathcal{O}_Z \lar 0,
$$
we obtain a long exact sequence
\begin{align}\label{eq:long-exact-seq-diff}
0 &\lar \Hom(\mathcal{I}_Z, \mathcal{I}_Z) \lar \Hom(\mathcal{I}_Z, \mathcal{O}_{\mathbb{P}^2}) \lar\\ &\lar \Hom(\mathcal{I}_Z, \mathcal{O}_{Z}) \xlongrightarrow{\partial} \Ext^1(\mathcal{I}_Z, \mathcal{I}_Z) \lar \Ext^1(\mathcal{I}_Z, \mathcal{O}_{\mathbb{P}^2}) \lar \ldots
\end{align}
Since $\mathcal I_Z$ is simple, \(\Hom(\mathcal I_Z,\mathcal I_Z)\simeq k\). We also claim that
\[
\Hom(\mathcal I_Z,\mathcal O_{\PP^2})\simeq k.
\]
Indeed, let
\[
\varphi:\mathcal I_Z\longrightarrow\mathcal O_{\PP^2}.
\]
On \(U=\PP^2\setminus Z\) the ideal sheaf $\mathcal I_Z$ is trivial, so $\restr {\varphi} {U}$ is
multiplication by a function in $H^0(U,\mathcal O_U)$. Since
$\PP^2$ is normal and $\PP^2\setminus U$ has codimension two, by Hartogs,
\[
H^0(\PP^2,\mathcal O_{\PP^2})
\xlongrightarrow{\sim}
H^0(U,\mathcal O_U),
\]
and hence \(\restr {\varphi} {U}=\lambda \cdot i\) for some $\lambda\in k$. Thus
$\varphi-\lambda \cdot i$ vanishes on the dense open set $U$. Its image is
therefore a torsion subsheaf of the torsion-free sheaf
$\mathcal O_{\PP^2}$, so it vanishes identically. Hence \(\varphi=\lambda \cdot i\), as claimed. So the first map in \eqref{eq:long-exact-seq-diff} is an isomorphism, and thus $\partial$ is injective. 

Now, let us start with a tangent vector $\alpha \in T_z Z_j$, representing a first-order deformation $J_\epsilon$ of the ideal $J$, and the associated $Z_\epsilon = V(J_\epsilon)$ deformation of the scheme $Z = V(J)$. By Theorem~\ref{thm:open-Z-j-quot-scheme}, a first-order deformation $J_\epsilon = (f_i + \epsilon g_i)$, corresponds to the tangent vector $\alpha = (g_1, g_2, g_3) \in \Hom(\T_J, \mathcal{I}_Z(d))$. By the deformation theory of the Quot functor, the extension class of $Z_\epsilon$, given by
$$
\alpha : 0 \lar \mathcal{I}_Z(d) \xlongrightarrow{\epsilon} \mathcal{I}_{Z_\epsilon}(d) \lar \mathcal{I}_Z(d) \lar 0
$$
in $\Ext^1(\mathcal{I}_Z(d), \mathcal{I}_Z(d))$ is the pushout of the sequence
$$
0 \lar \T_J \lar \mathcal{O}_{\mathbb{P}^2}^{\oplus 3} \lar \mathcal{I}_Z(d) \lar 0
$$
along $\alpha$, say 
\begin{equation}\label{eq:diagram-pushout-deformation}
\begin{tikzcd}
0 \arrow[r]
& \mathcal T_J
  \arrow[r,"i"]
  \arrow[d,"\alpha"]
& \mathcal O_{\PP^2}^{\oplus3}
  \arrow[r,"p"]
  \arrow[d]
& \mathcal I_Z(d)
  \arrow[r]
  \arrow[d,equal]
& 0 \\
0 \arrow[r]
& \mathcal I_Z(d)
  \arrow[r]
& E
  \arrow[r]
& \mathcal I_Z(d)
  \arrow[r]
& 0
\end{tikzcd}
\end{equation}
which coincides (see, for example, \cite[Proposition 4.4.4]{sernesi2006deformations}) with the image of 
$$
\alpha \in \Hom(\T_J, \mathcal{I}_Z(d))
$$ 
under the morphism
$$
\delta : \Hom(\T_J, \mathcal{I}_Z(d)) \lar \Ext^1(\mathcal{I}_Z(d), \mathcal{I}_Z(d)).
$$

The bottom row at \eqref{eq:diagram-pushout-deformation} represents the embedded first-order deformation $Z_\epsilon$ of
$Z\subset\PP^2$, and its associated extension class lies in the image of
the injective connecting morphism
\[
\partial:
\Hom(\mathcal I_Z,\mathcal O_Z)
\longrightarrow
\Ext^1(\mathcal I_Z,\mathcal I_Z).
\]
Moreover, under the natural isomorphism (see, for example, \cite[Chapter III, Proposition 6.7]{hartshorne2013algebraic})
\[
\tau:
\Ext^1(\mathcal I_Z(d),\mathcal I_Z(d))
\xlongrightarrow{\sim}
\Ext^1(\mathcal I_Z,\mathcal I_Z),
\]
the pushout description above gives
\begin{equation}\label{eq:diff-Hilb-j-diagram}
\partial\bigl(d_p\Phi_j(\alpha)\bigr)
=
\tau\bigl(\delta(\alpha)\bigr).    
\end{equation}
Thus the differential $d_p\Phi_j$ is characterized by the
commutative diagram in the statement.

Coming back to the sequence 
$$
0 \lar \T_J \lar \mathcal{O}_{\mathbb{P}^2}^{\oplus 3} \lar \mathcal{I}_Z(d) \lar 0
$$
and applying the functor $\Hom(-, \mathcal{I}_Z(d))$, we obtain a long exact sequence of the form
\begin{equation}\label{eq:long-exact-seq-Hilbert-diff}
0 \lar k \lar H^0(\mathcal{I}_Z(d))^{\oplus 3} \xlongrightarrow{\rho} T_p Z_j \xlongrightarrow{\delta} \Ext^1(\mathcal{I}_Z(d), \mathcal{I}_Z(d)) \xlongrightarrow{\psi} H^1(\mathcal{I}_Z(d))^{\oplus 3} \lar \ldots    
\end{equation}
From the equation \eqref{eq:diff-Hilb-j-diagram}, since $\partial$ is injective and $\tau$ is an isomorphism, we conclude
\[
d_p\Phi_j(\alpha)=0
\iff
\delta(\alpha)=0.
\]
Therefore \(\ker(d_p\Phi_j)=\ker(\delta)\).
By exactness of
\eqref{eq:long-exact-seq-Hilbert-diff}, \(\ker(\delta)=\im(\rho)\), and hence
\[
\ker(d_p\Phi_j)
=
\im(\rho)
\simeq
\frac{H^0(\mathcal I_Z(d))^{\oplus3}}
{k\cdot(f_1,f_2,f_3)}.
\]
The commutative diagram \eqref{eq:diff-Hilb-j-diagram} induces a well-defined morphism
\[
\overline{\partial}:
\coker(d_p\Phi_j)
\longrightarrow
\coker(\delta),
\qquad
[\xi]\longmapsto
[\tau^{-1}(\partial\xi)].
\]
Moreover, it is injective. Indeed, if
$\overline{\partial}([\xi])=0$, then there exists
$\alpha\in T_pZ_j$ such that
\[
\partial\xi=\tau(\delta(\alpha))
=\partial(d_p\Phi_j(\alpha)).
\]
Since $\partial$ is injective, \(
\xi=d_p\Phi_j(\alpha)\), and hence $[\xi]=0$ in $\coker(d_p\Phi_j)$. Therefore $\overline{\partial}$ is injective.

By exactness of
\eqref{eq:long-exact-seq-Hilbert-diff},
\[
\coker(\delta)\simeq\im(\psi)
\subseteq H^1(\mathcal I_Z(d))^{\oplus3},
\]
and consequently we obtain the injection of the claim
\[
\coker(d_p\Phi_j)
\hookrightarrow
H^1(\mathcal I_Z(d))^{\oplus3}.
\]

We have therefore obtained
\[
\ker(d_p\Phi_j)
\simeq
\frac{H^0(\mathcal I_Z(d))^{\oplus3}}
{k\cdot(f_1,f_2,f_3)} \text{ and }\coker(d_p\Phi_j)
\hookrightarrow
H^1(\mathcal I_Z(d))^{\oplus3}.
\]
In particular,
\[
\dim\ker(d_p\Phi_j)
=
3h^0(\mathcal I_Z(d))-1,
\]
and if \(h^1(\mathcal I_Z(d))=0\), then $d_p\Phi_j$ is surjective.
\end{proof}

The preceding proposition shows that the failure of the differential
$d_p\Phi_j$ to be surjective is controlled by the cohomology
\[
H^1(\mathcal I_Z(d)).
\]
This also gives the following useful smoothness criterion.

\begin{Lemma}\label{lem:obstruction-cohomology}
Let $d \geq 1$, $j \geq 0$ and $[J] \in Z_j$ with corresponding closed subscheme $Z = V(J) \subset \mathbb{P}^2$. Then:
\begin{itemize}
    \item[(a)] If $H^1(\mathcal{I}_Z(d)) = 0$, then $Z_j$ is smooth at $[J]$ and $\Phi_j$ is smooth morphism at $[J]$.
    \item[(b)] There is an isomorphism
    $$
    H^1(\mathcal{I}_Z(d)) \simeq H^0(\T_J(d-3))^*.
    $$
    In particular, whenever $e = \indeg(\syz(J)) \geq d-2$, then $h^1(\mathcal{I}_Z(d)) = 0$.
\end{itemize}
\end{Lemma}

\begin{proof}
By Theorem~\ref{thm:open-Z-j-quot-scheme},
\[
\Ext^1(\T_J,\mathcal I_Z(d))
\]
is an obstruction space for the Quot-scheme description of $Z_j$. Thus it suffices to show that this group vanishes under the
hypothesis. Applying the functor $\Hom(-, \mathcal{I}_Z(d))$ to the short exact sequence
$$
0 \lar \T_J \lar \mathcal{O}_{\mathbb{P}^2}^{\oplus 3} \lar \mathcal{I}_Z(d) \lar 0
$$
one gets a long exact sequence with the following piece:
\begin{equation}\label{eq:vanishing-obstruction-cohomology}
\ldots \lar H^1( \mathcal{I}_Z(d))^{\oplus 3} \lar \Ext^1(\T_J, \mathcal{I}_Z(d)) \lar \Ext^2(\mathcal{I}_Z(d), \mathcal{I}_Z(d)) \lar \ldots
\end{equation}
and we claim both endpoints above are zero. The left-hand one vanishes by hypothesis. For the right-hand group, since $\mathbb{P}^2$ is smooth, every coherent sheaf is perfect, and the derived form of Serre duality implies
$$
\Ext^2(\mathcal{I}_Z, \mathcal{I}_Z) \simeq \Hom(\mathcal{I}_Z, \mathcal{I}_Z(-3))^*.
$$

Any morphism $f: \mathcal{I}_Z \lar \mathcal{I}_Z(-3)$ induces, after taking reflexive hulls, a commutative diagram
\begin{center}
\begin{tikzcd}
\mathcal{I}_Z \arrow[r, "f"] \arrow[d]                & \mathcal{I}_Z(-3) \arrow[d]    \\
\mathcal{O}_{\mathbb{P}^2} \arrow[r, "f^{\vee \vee}"] & \mathcal{O}_{\mathbb{P}^2}(-3)
\end{tikzcd}
\end{center}
where the vertical arrows are injective. Note that $f^{\vee \vee} = 0$, and thus $f = 0$, showing the desired vanishing. 

Hence \( \Ext^1(\T_J,\mathcal I_Z(d))=0\). Since this group is an obstruction space for the Quot scheme
description of $Z_j$ from
Theorem~\ref{thm:open-Z-j-quot-scheme}, it follows that $Z_j$ is
smooth at $[J]$.

Moreover, by Proposition~\ref{prop:differential-Hilb-j}, the hypothesis \(h^1(\mathcal I_Z(d))=0\) implies that
\[
d_{[J]}\Phi_j:
T_{[J]}Z_j
\longrightarrow
T_{[Z]}\Hilbert^j(\PP^2)
\]
is surjective. By Fogarty's theorem,
$\Hilbert^j(\PP^2)$ is smooth, and we have just proved that $Z_j$ is
smooth at $[J]$. Therefore the standard tangent-space criterion for
smoothness implies that $\Phi_j$ is smooth at $[J]$. This proves~(a).

For~(b), cohomology applied to the short exact sequence
$$
0 \lar \T_J \lar \mathcal{O}_{\mathbb{P}^2}^{\oplus 3} \lar \mathcal{I}_Z(d) \lar 0
$$
yields an isomorphism
$$
H^1(\mathcal{I}_Z(d)) \simeq H^2(\T_J),
$$
and Serre duality on $\mathbb{P}^2$ gives
$$
H^2(\T_J) \simeq H^0(\T_J^\vee(-3))^*.
$$
Since $\T_J$ is a rank two vector bundle with $\det(\T_J) \simeq \mathcal{O}_{\mathbb{P}^2}(-d)$, there is an isomorphism $\T_J^\vee \simeq \T_J(d)$, and thus we obtain the desired isomorphism
$$
H^1(\mathcal{I}_Z(d)) \simeq H^2(\T_J) \simeq H^0(\T_J^\vee(-3))^* \simeq H^0(\T_J(d-3))^*.
$$
\end{proof}

A standard residual argument gives the following useful consequence.

\begin{Proposition}\label{prop:Z-smooth-irr-low-j}
Let $d\geq2$ and $0\leq j\leq d$. If $Z_j$ is nonempty, then
$Z_j\subset\PP(W_d)$ is smooth and irreducible of codimension $j$,
and
\[
\Phi_j:Z_j\longrightarrow\Hilbert^j(\PP^2)
\]
is a smooth morphism with dense open image.
\end{Proposition}

\begin{proof}
We first recall that, for every zero-dimensional subscheme
$Z\subset\PP^2$ of length $j$,
\begin{equation}\label{eq:postulation-small-length}
H^1(\mathcal I_Z(m))=0
\qquad\text{for every }m\geq j-1.
\end{equation}
We prove this by induction on $j$. The case $j=0$ is immediate. For $j>0$, choose a line $L$ meeting $Z$, set
\[
k\doteq\length(Z\cap L)\geq1,
\]
and let
\(
Z'\doteq\operatorname{Res}_L(Z)\). Then \(\length(Z')=j-k\) and the residual sequence gives
\[
0
\longrightarrow
\mathcal I_{Z'}(m-1)
\longrightarrow
\mathcal I_Z(m)
\longrightarrow
\mathcal I_{Z\cap L,L}(m)
\longrightarrow0.
\]
Since $L\simeq\PP^1$ and $Z\cap L$ has length $k$,
\[
\mathcal I_{Z\cap L,L}(m)
\simeq
\mathcal O_L(m-k).
\]
If $m\geq j-1$, then
\[
m-1\geq (j-k)-1
\qquad\text{and}\qquad
m-k\geq-1.
\]
Thus, by induction,
\[
H^1(\mathcal I_{Z'}(m-1))=0,
\]
while
\[
H^1(\mathcal O_L(m-k))=0.
\]
The cohomology sequence proves
\eqref{eq:postulation-small-length}.

Now let $[J]\in Z_j$. Since $j\leq d$, we have \(H^1(\mathcal I_Z(d))=0\). Lemma~\ref{lem:obstruction-cohomology} therefore shows that $Z_j$
and $\Phi_j$ are smooth. Since $Z_j$ is nonempty,
Lemma~\ref{lem:irreducibility-from-Hilbert-map} shows that $Z_j$ is
irreducible and that $\Phi_j(Z_j)$ is dense open in
$\Hilbert^j(\PP^2)$.

Finally, Lemma~\ref{lem:dimension-Z-j-smooth-points} gives
\[
\dim Z_j
=
3\binom{d+2}{2}-j-1
=
\dim\PP(W_d)-j.
\]
Hence $Z_j$ has codimension $j$ in $\PP(W_d)$.
\end{proof}

We finish the section by recording the complete picture for
quadratic triples.

\begin{Proposition}\label{prop:d-2-Phi-j}
Assume $d=2$. Then:

\begin{itemize}
\item[(a)] $Z_1,Z_2,Z_3$ are smooth and irreducible of dimensions
$16,15,14$, respectively;

\item[(b)] the morphisms $\Phi_j$ are smooth for $j=1,2,3$;

\item[(c)] $\Phi_1$ and $\Phi_2$ are surjective, while
\[
\Phi_3(Z_3)
=
\left\{
[Z]\in\Hilbert^3(\PP^2):
Z\text{ is not contained in a line}
\right\}.
\]
\end{itemize}
\end{Proposition}

\begin{proof}
For $j=1,2$, the smoothness, irreducibility, and dimension statements
follow directly from Proposition~\ref{prop:Z-smooth-irr-low-j}.

To see that $\Phi_1$ is surjective, after a projective change of
coordinates any point may be written as $p=[0:0:1]$, and the triple
\[
xz+y^2,\qquad yz+x^2,\qquad xy
\]
has base scheme $\{p\}$.

For $j=2$, every length-two subscheme $Z\subset\PP^2$ is contained
in a line. After a change of coordinates,
\[
I_Z=(x,q(y,z)),
\]
where $q$ is quadratic. Then
\[
x^2+q,\qquad xy,\qquad xz
\]
are linearly independent quadrics generating $\mathcal I_Z$ as an
ideal sheaf. Hence $\Phi_2$ is surjective.

Now let $j=3$ and put
\[
B_3
\doteq
\left\{
[Z]\in\Hilbert^3(\PP^2):
H^0(\mathcal I_Z(1))=0
\right\}.
\]
This is the open locus of length-three schemes not contained in a
line. If $[J]\in Z_3$ and $Z=V(J)$ were contained in a line $L$,
then every quadric vanishing on $Z$ would vanish identically on $L$.
Thus the three generators of $J$ would have the equation of $L$ as a
common factor, contradicting $[J]\in U_2$. Hence \(\Phi_3(Z_3)\subseteq B_3\).

Conversely, if $[Z]\in B_3$, then the Hilbert function of $Z$ is
\[
1,3,3,\ldots,
\]
so \(h^0(\mathcal I_Z(2))=3\) and $\mathcal I_Z$ is generated by three quadrics. Choosing a basis
gives a point of $Z_3$, and therefore \(
\Phi_3(Z_3)=B_3\).

For $Z\in B_3$ one also has \(H^1(\mathcal I_Z(2))=0\). Thus Lemma~\ref{lem:obstruction-cohomology} shows that $Z_3$ and
$\Phi_3$ are smooth. Since
\(
\Phi_3(Z_3)=B_3\neq\emptyset\), Lemma~\ref{lem:irreducibility-from-Hilbert-map} shows that $Z_3$ is
irreducible. Finally,
\[
\dim Z_3=17-3=14
\]
by Lemma~\ref{lem:dimension-Z-j-smooth-points}.
\end{proof}

\section{Mapping triples to associated Bourbaki schemes}\label{sec:Bourbaki-map}

In this section, we focus further on the locally closed strata
\[
I_e^j
=
\left\{
[J]\in Z_j:
\indeg(\syz(J))=e
\right\},
\]
on which the Bourbaki degree is constant.

We start with a universal Bourbaki construction, described in Lemma~\ref{lem:universal-Bourbaki-construction}, which will be used to produce a morphism of schemes (in Theorem~\ref{thm:morphism-Bourbaki-scheme}):
$$
\nu \in \syz(J)_e \longmapsto B = B(\nu, J) \subset \mathbb{P}^2,
$$
from an appropriate parameter space of pairs to the Hilbert scheme $\Hilbert^b(\mathbb{P}^2)$, where $B = B(\nu, J)$ is the Bourbaki scheme associated to $J$ and a choice of non-zero syzygy $\nu$. 

Afterwards, in Proposition~\ref{prop:free-strata-smooth}, we show the free strata $I_e^{d(d-e)+e^2}$ are smooth and irreducible varieties, giving a closed formula for their dimension. 

\begin{Lemma}\label{lem:universal-Bourbaki-construction}
Fix $d\geq1$, $j\geq1$, and $1\leq e\leq d$. Let \(W_e^j\subseteq I_e^j\) be the open locus on which
\(h^0(\T_J(e))\) takes its minimum value, say $r_e$, and endow $W_e^j$ with its reduced induced scheme
structure.

Let \(
q:\PP^2\times W_e^j\longrightarrow W_e^j\) be the projection and let $\T_W$ denote the restriction of the
universal syzygy bundle to $\PP^2\times W_e^j$. Then:

\begin{enumerate}
    \item[(a)] The sheaf \(\mathcal E_e^j \doteq q_*(\T_W(e))\)
    is locally free of rank $r_e$, its formation commutes with base
    change, and for every $[J] \in W_e^j$ there is a natural isomorphism 
    \[
    \mathcal E_e^j\otimes k([J])
    \simeq
    H^0(\PP^2,\T_J(e))
    \simeq
    \syz(J)_e.
    \]

    \item[(b)] Let \(\pi: \PP(\mathcal E_e^j)
    \longrightarrow W_e^j\) be the induced projective bundle by $\mathcal E_e^j$.
    Then the fiber over $[J]$ is naturally isomorphic to \(\pi^{-1}([J])\simeq\PP(\syz(J)_e)\).

    \item[(c)] There exists a finite flat family \(\mathcal B\subseteq\PP^2\times \PP(\mathcal{E}_e^j)\) of degree
    \[
    b=e(e-d)+d^2-j
    \]
    such that, for every geometric point $y=([J],[\nu])\in\PP(\mathcal E_e^j)$,
    the fiber $\mathcal B_y$ is the Bourbaki scheme
    $B_\nu(J)$ associated with the minimal syzygy $\nu$.
\end{enumerate}
\end{Lemma}

\begin{proof}
Endow the open subset $W = W_e^j \subset I_e^j$ with its reduced
induced scheme structure. Let $\T_W$ denote the restriction of the
universal syzygy bundle to $\PP^2\times W$. Since $\T_W(e)$ is
coherent and flat over $W$, Grauert's theorem
\cite[Chapter~III, Corollary~12.9]{hartshorne2013algebraic}
implies that \(
\mathcal E = \mathcal{E}_e^j \doteq q_*\T_W(e) \)
is a locally free sheaf of rank $r_e$. 

For a point $s=[J]\in W$, consider the Cartesian square
\[
\begin{tikzcd}
\PP^2_{k(s)}
\arrow[r,"\widetilde{i}_s"]
\arrow[d,"q_s"']
&
\PP^2\times W_e^j
\arrow[d,"q"]
\\
\Spec k(s)
\arrow[r,"i_s"']
&
W_e^j .
\end{tikzcd}
\]
By cohomology and base change
\cite[Chapter~III, Corollary~12.9 and Theorem~12.11]{hartshorne2013algebraic}, the natural base-change morphism
\[
i_s^*q_*\T_W(e)
\longrightarrow
q_{s*}\widetilde{i}_s^*\T_W(e)
\]
is an isomorphism. Since
\[
i_s^*q_*\T_W(e)
=
(q_*\T_W(e))\otimes_{\mathcal O_W}k(s)
\]
and \(\widetilde{i}_s^*\T_W(e)
\simeq
\T_J(e)\),
we obtain
\begin{equation}\label{eq:univ-Bourbaki-fibers}
\mathcal E_e^j\otimes k([J])
\simeq
H^0\bigl(\PP^2,\T_J(e)\bigr)
\simeq
\syz(J)_e.
\end{equation}
This proves~(a).

We use the convention that
\(\PP_W(\mathcal E)\) parametrizes one-dimensional subspaces of the fibers of $\mathcal E$. Thus, if \(
\pi:\PP_W(\mathcal E)\longrightarrow W
\) denotes the projection, there is a tautological inclusion
\[
\mathcal O_{\PP(\mathcal E)}(-1)
\hookrightarrow
\pi^*\mathcal E.
\]
For every $[J]\in W$, the preceding base-change isomorphism in \eqref{eq:univ-Bourbaki-fibers} gives
\[
\pi^{-1}([J])
\simeq
\PP\bigl(
H^0(\PP^2,\T_J(e))
\bigr)
=
\PP(\syz(J)_e).
\]
Thus a point of $\PP(\mathcal E)$ is a pair \(([J],[\nu])\) where \([J] \in W_{e}^j\) and \(0\neq\nu\in\syz(J)_e\),
where $[\nu]$ is considered up to multiplication by a nonzero scalar. This shows~(b).

Let \(X\doteq\PP^2\times \PP(\mathcal E)\) and denote the projections by
\(
p:X\lar\PP^2\) and \(
q_\mathcal{E}:X\lar \PP(\mathcal E).
\)
Set
\[
r\doteq\id_{\PP^2}\times\pi:
X\longrightarrow\PP^2\times W.
\]
There is a canonical evaluation morphism
\[
q^*\mathcal E
=
q^*q_*\T_W(e)
\longrightarrow
\T_W(e).
\]
Pulling it back by $r$ and composing with the pullback of the
tautological inclusion gives
\[
{q_\mathcal{E}}^*\mathcal O_{\PP(\mathcal E)}(-1)
\longrightarrow
{q_\mathcal{E}}^*\pi^*\mathcal E
=
r^*q^*\mathcal E
\longrightarrow
r^*\T_W(e).
\]
After twisting by $p^*\mathcal O_{\PP^2}(-e)$, we obtain the
universal syzygy morphism
\[
\iota:
p^*\mathcal O_{\PP^2}(-e)
\otimes {q_\mathcal{E}}^*\mathcal O_{\PP(\mathcal E)}(-1)
\longrightarrow
r^*\T_W.
\]
Set
\[
\mathcal L
\doteq
p^*\mathcal O_{\PP^2}(-e)
\otimes
q_{\mathcal E}^*\mathcal O_{\PP(\mathcal E)}(-1),
\qquad
\mathcal T_X
\doteq
r^*\mathcal T_W.
\]
Then the universal syzygy morphism takes the form
\[
\iota:\mathcal L\longrightarrow\mathcal T_X.
\]
If \(
y=([J],[\nu])\in \PP(\mathcal E)\), then, under the natural identification \eqref{eq:univ-Bourbaki-fibers}, the fiber of the tautological line at $y$ is precisely the line \(k \cdot \nu\subseteq H^0(\PP^2,\T_J(e))\).
Consequently, the restriction of $\iota$ to
$\PP^2\times\{y\}$ is exactly the morphism
\[
\iota_y:
\mathcal O_{\PP^2}(-e)
\xlongrightarrow{\nu}
\T_J
\]
defined by the chosen minimal syzygy $\nu$. This morphism is injective. Indeed, $\nu$ is nonzero, and hence
generically nonzero. Moreover, if the coefficients of $\nu$ had a
nonconstant common factor, dividing by this factor would produce a
syzygy of degree strictly smaller than $e$, contradicting
\(e=\indeg(\syz(J))\). Thus $\nu$ has no divisorial zero, and its cokernel is a torsion-free
rank-one sheaf.

Since $\T_X = r^*(\T_W)$ is locally free, hence flat over $\mathbb{P}(\mathcal{E})$, and every geometric fiber of $\iota$ is injective, Lemma~\ref{lem:flatness-finite-cokernel} shows that $\iota$ is
injective and the cokernel \(\mathcal C\doteq\coker(\iota)\) is flat over $\mathbb{P}(\mathcal{E})$. Thus
\[
0\longrightarrow
\mathcal L
\xlongrightarrow{\iota}
\T_X
\longrightarrow
\mathcal C
\longrightarrow0
\]
is exact.

We now apply the determinant construction, as in Lemma~\ref{lem:compatibility-determinant}. Set \( \mathcal D \doteq \det(\T_X)\otimes\mathcal L^\vee\). The wedge pairing
\[
\mathcal L\otimes\T_X
\longrightarrow
\det(\T_X)
\]
induces a morphism \(\T_X\longrightarrow\mathcal D\) which annihilates $\mathcal L$. Hence it factors uniquely through $\mathcal C$, say by \(\bar\delta:
\mathcal C\longrightarrow\mathcal D\).

On the fiber over $y=([J],[\nu])$, this is precisely the canonical
inclusion
\[
\mathcal I_{B_\nu(J)}(e-d)
\hookrightarrow
\mathcal O_{\PP^2}(e-d).
\]
Indeed, \(\det(\T_J)\simeq\mathcal O_{\PP^2}(-d)\), so \(\mathcal D_y
\simeq
\mathcal O_{\PP^2}(e-d)\),
and the quotient is supported on the Bourbaki scheme $B_\nu(J)$.
Its length is
\[
\length B_\nu(J)
=
e(e-d)+d^2-j
=
b
\]
by the Bourbaki degree formula. Thus every geometric fiber of \(\bar\delta:\mathcal C\lar\mathcal D
\) is injective and has cokernel of length $b$. Since $\mathcal D$ is
invertible, another application of
Lemma~\ref{lem:flatness-finite-cokernel} shows that
\[
\mathcal I_{\mathcal B}
\doteq
\mathcal C\otimes\mathcal D^\vee
\hookrightarrow
\mathcal O_X
\]
defines a finite flat family \(\mathcal B\subseteq\PP^2\times \mathbb{P}(\mathcal{E})
\) of degree $b$ over $\mathbb{P}(\mathcal{E})$. By construction, \(
\mathcal B_y=B_\nu(J)\) for every geometric point $y=([J],[\nu])$.
\end{proof}

\begin{Theorem}\label{thm:morphism-Bourbaki-scheme}
Fix $d\geq1$, $j\geq1$, and $1\leq e\leq d$, and let
$W_e^j$ and $\mathcal E_e^j$ be as in
Lemma~\ref{lem:universal-Bourbaki-construction}. Then there is a
morphism
\[
\begin{aligned}
\varphi_{j,e}:\PP(\mathcal E_e^j)
&\longrightarrow
\Hilbert^b(\PP^2),\\
([J],[\nu])
&\longmapsto
[B_\nu(J)],
\end{aligned}
\]
where \(b=e(e-d)+d^2-j\).
\end{Theorem}

\begin{proof}
By Lemma~\ref{lem:universal-Bourbaki-construction}, there is a finite
flat family
\(
\mathcal B\subseteq\PP^2\times \mathbb{P}(\mathcal{E}_e^j)
\)
whose geometric fibers have constant length \(b=e(e-d)+d^2-j\). By the universal property of the Hilbert scheme, this family induces
a unique morphism
\[
\varphi_{j,e}:
\mathbb{P}(\mathcal{E}_e^j)
\longrightarrow
\Hilbert^b(\PP^2).
\]
For every geometric point \(
([J],[\nu])\in \mathbb{P}(\mathcal{E}_e^j)
\), the construction of $\mathcal B$ gives \(
\varphi_{j,e}([J],[\nu])=[B_\nu(J)]\), as claimed.
\end{proof}

In some cases the minimum value $r_e$ is attained on the whole
stratum, so that set-theoretically $W_e^j=I_e^j$. If moreover
$r_e=1$, then
\[
\PP(\mathcal E_e^j)\simeq W_e^j=(I_e^j)_{\mathrm{red}},
\]
and the universal Bourbaki construction gives a morphism
\[
(I_e^j)_{\mathrm{red}}
\longrightarrow
\Hilbert^b(\PP^2).
\]
Whenever $I_e^j$ is known to be reduced (for example, the free case, by Proposition~\ref{prop:free-strata-smooth} below), this is a morphism \(I_e^j\longrightarrow\Hilbert^b(\PP^2)\).

\begin{Proposition}\label{prop:generic-initial-degree}
Let $d \geq 1$ and $1 \leq j \leq d^2$. If $Z_j$ is irreducible, then there exists a unique $e_0 \in \{1, \ldots, d\}$ such that $I_{e_0}^j \subset Z_j$ is a nonempty dense open subset. Moreover,
$$
e_0 = \max\{ \indeg(\T_{[f]}): [f] \in Z_j \}.
$$
\end{Proposition}

\begin{proof}
Since the finitely many strata $I_e^j$ cover $Z_j$, there exists a
maximal integer
\[
e_0
=
\max\left\{
e:I_e^j\neq\emptyset
\right\}.
\]
By Lemma~\ref{lem:indeg-lowersemicontinuous}, the locus
\[
\left\{
[J]\in Z_j:
\indeg(\T_J)>e_0-1
\right\}
\]
is open in $Z_j$. By maximality of $e_0$, this locus is precisely
$I_{e_0}^j$. Hence $I_{e_0}^j$ is a nonempty open subset of the
irreducible scheme $Z_j$, and therefore is dense.

If another stratum $I_e^j$, with $e\neq e_0$, were dense, it would
intersect the nonempty open subset $I_{e_0}^j$, contradicting the
disjointness of the initial-degree strata. Thus $I_{e_0}^j$ is the
unique dense stratum.

By its definition, \(e_0=\max\left\{
\indeg(\T_J):[J]\in Z_j
\right\}\), and the result follows.
\end{proof}

We illustrate this phenomenon with some examples:

\begin{Example}
Return to Example~\ref{ex:subscheme-flat-and-quot-flat}. It defines
a morphism
\[
\gamma:\mathbb A^1\longrightarrow Z_6.
\]
At $t=0$, \((0,y,-3z) \in \syz(J_0)_1\), so \(\gamma(0)\in I_1^6\). For $t\neq0$ a direct coefficient comparison shows that no nonzero
linear syzygy exists, while
\[
(0,-y^2,2tx^2+3yz)
\]
is a quadratic syzygy. Hence
\(\gamma(t)\in I_2^6\) for $t \neq 0$. Thus
\[
\gamma(0)\in
\overline{I_2^6}\cap I_1^6,
\]
so $I_2^6$ is not closed in $Z_6$.

Moreover, \(\Bour(J_t)=1\) for every $t$, and the two minimal resolutions are
\[
0\lar R(-3)\lar R(-2)^3\lar\syz(J_t)\lar0
\qquad(t\neq0),
\]
and
\[
0\lar R(-4)\lar
R(-1)\oplus R(-3)^2
\lar\syz(J_0)\lar0.
\]
\end{Example}

The present study raises the following question, which we answer for the free case.

\begin{Question}
Is there a closed formula for
\(\dim I_e^j\)
in terms of $d,e$, and $j$?
\end{Question}

For $Z_j$, Theorem~\ref{thm:open-Z-j-quot-scheme} provides an
obstruction space, while
Lemma~\ref{lem:dimension-Z-j-smooth-points} gives a closed dimension
formula at unobstructed points. For the finer strata $I_e^j$, a
dimension formula is presently available in particular when $e$ is
the generic, equivalently maximal, initial degree on an irreducible
stratum $Z_j$, since then $I_e^j$ is dense open in $Z_j$ by
Proposition~\ref{prop:generic-initial-degree}.

For fixed degree $d \geq 1$, the strata $I_e^{d(d-e)+e^2}$ coincide with the locus of free ideals with splitting $\mathcal{O}_{\mathbb{P}^2}(-e) \oplus \mathcal{O}_{\mathbb{P}^2}(e-d)$. In the following proposition, we use a geometric argument to compute its dimension.

\begin{Proposition}\label{prop:free-strata-smooth}
Let $d\geq1$, \(
1\leq e\leq \lfloor d/2 \rfloor\) and \(j=d(d-e)+e^2\). Then $I_e^j$ is a smooth irreducible variety. Its dimension is
\[
\dim I_e^j
=
\begin{cases}
\displaystyle
3\left[
\binom{e+2}{2}
+
\binom{d-e+2}{2}
\right]
-
\left[
2+\binom{d-2e+2}{2}
\right],
& e<d/2,\\[4mm]
\displaystyle
6\binom{d/2+2}{2}-4,
& e=d/2.
\end{cases}
\]
Equivalently,
\[
\dim I_e^j
=
\begin{cases}
d^2-de+e^2+3d+3e+3,
& e<d/2,\\[2mm]
\displaystyle
\frac34d^2+\frac92d+2,
& e=d/2.
\end{cases}
\]
\end{Proposition}

\begin{proof}
Set \( \T \doteq
\mathcal O_{\PP^2}(-e)
\oplus
\mathcal O_{\PP^2}(e-d)
\) and
\[
V\doteq H^0(\T^\vee)
=
H^0(\mathcal O_{\PP^2}(e))
\oplus
H^0(\mathcal O_{\PP^2}(d-e)).
\]
Thus \(\Hom(\T,\mathcal O_{\PP^2}^{\oplus3}) \simeq V^{\oplus3}\). Let \( P\doteq\PP(V^{\oplus3})\).
A point $[\varphi]\in P$ may be represented by a matrix
\[
\varphi=
\begin{pmatrix}
a_{11}&a_{12}\\
a_{21}&a_{22}\\
a_{31}&a_{32}
\end{pmatrix},
\]
where
\(
a_{i1}\in H^0(\mathcal O_{\PP^2}(e))\) and \(a_{i2}\in H^0(\mathcal O_{\PP^2}(d-e))\). 

Let $U\subset P$ be the locus where $\varphi$ has generic rank two
and the cokernel
\[
\mathcal C_\varphi\doteq\coker(\varphi)
\]
is torsion-free. Equivalently, the three maximal minors of $\varphi$
have no common divisorial factor. These are open conditions.
Consequently $U$ is a nonempty open subset of the projective space
$P$, and hence is smooth and irreducible, with
\[
\dim U
=
3\dim V-1.
\]

For $[\varphi]\in U$, the morphism
\[
\varphi:\T\longrightarrow\mathcal O_{\PP^2}^{\oplus3}
\]
is injective. Indeed, it is generically injective, and its kernel is
a torsion subsheaf of the locally free sheaf $\T$, hence is zero.
The cokernel is therefore a torsion-free rank-one sheaf with
determinant $\mathcal O_{\PP^2}(d)$, and hence \(\mathcal C_\varphi
\simeq
\mathcal I_Z(d)\)
for the zero-dimensional subscheme $Z=V(\Delta_{12}(\varphi), - \Delta_{13}(\varphi), \Delta_{23}(\varphi))\subset\PP^2$.

The Hilbert polynomial is determined by the exact sequence
\[
0\longrightarrow
\T
\xlongrightarrow{\varphi}
\mathcal O_{\PP^2}^{\oplus3}
\longrightarrow
\mathcal I_Z(d)
\longrightarrow0.
\]
A direct Chern class computation gives \(
\deg(Z)=d(d-e)+e^2=j\). Thus the maximal minors of $\varphi$ determine a point of $I_e^j$,
and we obtain a surjective morphism
\[
\pi:U\longrightarrow I_e^j.
\]

We next describe its fibers. Let
\[
G\doteq
\Aut(\T)/(k^\times \cdot \Id_{\T}).
\]
Precomposition gives an action of $G$ on $U$, and $\pi$ is
$G$-invariant. Moreover, two points of $U$ have the same image in
$I_e^j$ if and only if they differ by this action. Indeed, two
matrices determining the same triple of maximal minors give two
Hilbert--Burch resolutions of the same ideal, and uniqueness of the
minimal free resolution yields an automorphism of $\T$ relating the
two matrices. Hence the fibers of $\pi$ are precisely the
$G$-orbits.

We now show that $\pi$ is smooth. Let $\mathscr T$ denote the
universal syzygy bundle on \(\PP^2\times I_e^j\), and let
\[
p:\PP^2\times I_e^j\lar\PP^2,
\qquad
q:\PP^2\times I_e^j\lar I_e^j
\]
be the projections. Every geometric fiber of $\mathscr T$ is
isomorphic to $\T$. Since \(H^1(\mathcal End(\T))=0\), cohomology and base change
\cite[Chapter~III, Theorem~12.11]{hartshorne2013algebraic}
implies that
\[
\mathcal H
\doteq
q_*\mathcal Hom(p^*\T,\mathscr T)
\]
is locally free and that its formation commutes with base change.
In particular,
\[
\mathcal H\otimes k([J])
\simeq
\Hom(\T,\T_J)
\]
for every $[J]\in I_e^j$.

Let \(\mathscr P\subseteq\PP_{I_e^j}(\mathcal H)\) be the open locus parametrizing fiberwise isomorphisms
\[
\alpha:\T\xlongrightarrow{\sim}\T_J
\]
up to a nonzero scalar. The projection \(\rho:\mathscr P\longrightarrow I_e^j
\) is smooth, since $\mathscr P$ is open in a projective bundle, and it
is surjective because $\T_J\simeq\T$ for every $[J]\in I_e^j$.

We claim that $\mathscr P$ is naturally isomorphic to the open set
$U\subset\PP\Hom(\T,\mathcal O_{\PP^2}^{\oplus3})$ introduced
above. Indeed, over a point $[J]\in I_e^j$, composition with the
universal syzygy inclusion
\[
\jmath_J:\T_J\hookrightarrow\mathcal O_{\PP^2}^{\oplus3}
\]
sends an isomorphism
\[
\alpha:\T\xlongrightarrow{\sim}\T_J \text{ to } \jmath_J\circ\alpha:
\T\longrightarrow\mathcal O_{\PP^2}^{\oplus3}.
\]
Its cokernel is $\mathcal I_Z(d)$, so it determines a point of $U$. Conversely, if $[\varphi]\in U$, then
\[
0\longrightarrow\T
\xlongrightarrow{\varphi}
\mathcal O_{\PP^2}^{\oplus3}
\longrightarrow\mathcal I_Z(d)
\longrightarrow0
\]
is exact. Hence
\[
\im(\varphi)
=
\ker\bigl(
\mathcal O_{\PP^2}^{\oplus3}
\lar\mathcal I_Z(d)
\bigr)
=
\T_J,
\]
and $\varphi$ induces an isomorphism \(\T\xlongrightarrow{\sim}\T_J\). These constructions are inverse and are compatible with families, so \(\mathscr P\simeq U \)
over $I_e^j$. Consequently,
\[
\pi:U\longrightarrow I_e^j
\]
is smooth and surjective.

Since $U$ is an open subset of a projective space, it is smooth and
irreducible. To see that $I_e^j$ is smooth, let $u\in U$ map to
$y\in I_e^j$, and let $r$ be the relative dimension of $\pi$ at
$u$. Smoothness of $\pi$ gives
\[
0\longrightarrow T_u(U_y)
\longrightarrow T_uU
\longrightarrow T_yI_e^j
\longrightarrow0,
\]
with \( \dim T_u(U_y)=r\). Since $U$ is smooth and
\[
\dim_uU=\dim_yI_e^j+r,
\]
we obtain
\[
\dim T_yI_e^j
=
\dim T_uU-r
=
\dim_uU-r
=
\dim_yI_e^j.
\]
Thus $I_e^j$ is smooth at $y$. Since $y$ was arbitrary,
$I_e^j$ is smooth. Moreover, $\pi$ is surjective and $U$ is irreducible, so
$I_e^j$ is irreducible.

It remains to compute the dimension. If $e<d/2$, then
\[
\Aut(\T)
=
\left\{
\begin{pmatrix}
a&\psi\\
0&b
\end{pmatrix}
:
a,b\in k^\times,\ 
\psi\in H^0(\mathcal O_{\PP^2}(d-2e))
\right\},
\]
and therefore
\[
\dim\Aut(\T)
=
2+\binom{d-2e+2}{2}.
\]
If $e=d/2$, then
\(
\T\simeq\mathcal O_{\PP^2}(-e)^{\oplus2}\) and \(\Aut(\T)\simeq\GL_2(k)\),
so \(\dim\Aut(\T)=4\).

Since \(\dim G=\dim\Aut(\T)-1\)
and the fibers of $\pi$ have dimension $\dim G$, we obtain
\[
\dim I_e^j
=
\dim U-\dim G
=
3\dim V-\dim\Aut(\T).
\]
Substituting
\[
\dim V
=
\binom{e+2}{2}+\binom{d-e+2}{2}
\]
gives the first displayed formula. Simplifying yields
\[
\dim I_e^j
=
d^2-de+e^2+3d+3e+3
\]
when $e<d/2$, and
\[
\dim I_{d/2}^j
=
\frac34d^2+\frac92d+2
\]
when $e=d/2$.
\end{proof}

Comparing with the previous results about free points, we get the following corollary:

\begin{Corollary}\label{cor:free-points-picture}
Let $d \geq 1$, $1 \leq e \leq \lfloor d/2 \rfloor$ and $j = d(d-e)+e^2$. Then the natural locally closed immersion of the free stratum 
$$
I_e^j \hookrightarrow Z_j
$$
is an open immersion. The scheme $I_e^j$ is smooth and irreducible, and $Z_j$ is smooth along $I_e^j$.

Furthermore, the closure $\overline{I_e^j} \subset Z_j$ is an irreducible component of $Z_j$. Hence $I_e^j$ is a smooth dense open subset of $\overline{I_e^j}$.
\end{Corollary}

\begin{proof}
Let $[J]\in I_e^j$. In Proposition~\ref{prop:dim-at-free-points}, we compute the dimension \(\dim T_{[J]}Z_j\) and, after substituting $j = d(d-e)+e^2$, the computed formula agrees with the dimensional formula from Proposition~\ref{prop:free-strata-smooth}. Hence \(\dim T_{[J]}Z_j = \dim I_e^j\).

Since $I_e^j$ is locally closed in $Z_j$, there exists an open
neighborhood $U\subset Z_j$ of $[J]$ such that
$U\cap I_e^j$ is closed in $U$. Put $A\doteq \mathcal O_{Z_j,[J]}$ and 
$B\doteq \mathcal O_{I_e^j,[J]}$. The corresponding closed immersion induces a surjection \(A\twoheadrightarrow B\). Hence \(\dim A\geq \dim B=\dim I_e^j\).

On the other hand, for a Noetherian local $k$-algebra of finite type,
\[ 
\dim A \leq \dim_k\mathfrak m_A/\mathfrak m_A^2
=
\dim T_{[J]}Z_j.
\]
Since \(\dim T_{[J]}Z_j=\dim I_e^j\), we obtain
\( \dim A = \dim I_e^j = \dim T_{[J]}Z_j.
\)
Thus $A$ is a regular local ring. In particular, $Z_j$ is smooth at
$[J]$.

We claim moreover that the surjection $A\twoheadrightarrow B$ is an
isomorphism. Let
\[
K=\ker(A\lar B).
\]
If $K\neq0$, choose $0\neq f\in K$. Since the regular local ring $A$
is a domain, $f$ is a non-zero-divisor, and therefore \(\dim A/(f)=\dim A-1\). As $B=A/K$ is a quotient of $A/(f)$, this would give
\(\dim B\leq \dim A-1\), contradicting $\dim B=\dim A$. Hence $K=0$, and therefore \(A\simeq B\).

It follows in fact that \(I_e^j\hookrightarrow Z_j\) is an open immersion. Indeed, locally on $U$ the closed subscheme
$U\cap I_e^j$ is defined by a coherent ideal sheaf $\mathcal K$.
The isomorphism of local rings above gives \(\mathcal K_{[J]}=0\). Hence $\mathcal K$ vanishes on an open neighborhood of $[J]$, so
$I_e^j$ contains an open neighborhood of each of its points in
$Z_j$. Consequently $I_e^j$ is open in $Z_j$.

By Proposition~\ref{prop:free-strata-smooth}, $I_e^j$ is
irreducible, and hence its closure \(\overline{I_e^j}\) is irreducible. Choose $[J]\in I_e^j$. Since $Z_j$ is smooth at
$[J]$, there is a unique irreducible component $C$ of $Z_j$ passing
through $[J]$. Moreover,
\[
\dim C
=
\dim\mathcal O_{Z_j,[J]}
=
\dim I_e^j.
\]
Since $\overline{I_e^j}$ is irreducible and contains $[J]$, it is
contained in $C$. But
\[
\dim\overline{I_e^j}
=
\dim I_e^j
=
\dim C.
\]
A proper closed subset of the irreducible variety $C$ has strictly
smaller dimension, so \(\overline{I_e^j}=C\). Thus $\overline{I_e^j}$ is an irreducible component of $Z_j$, and
$I_e^j$ is a smooth dense open subset of this component.
\end{proof}

Although the free stratum $I_e^j$ is smooth, hence reduced, it need not be closed in $Z_j$. The following example exhibits this phenomenon explicitly.

\begin{Example}\label{ex:1-par-family-free-degeneration}
There exists a pointed integral curve $(B,0)$ and a morphism
\[
\gamma:B\longrightarrow Z_{27}
\]
such that
\[
\gamma(t)\in I_3^{27}
\quad(t\neq0),
\qquad
\gamma(0)\in I_2^{27}.
\]
In particular \(\gamma(0)\in
\overline{I_3^{27}}\setminus I_3^{27}\).

Fix $p\in\PP^2$ and let $F$ be a nontrivial locally free extension
\begin{equation}\label{eq:special-nonsplit-family}
0
\longrightarrow
\mathcal O_{\PP^2}(1)
\longrightarrow
F
\longrightarrow
\mathcal I_p(-1)
\longrightarrow0.
\end{equation}
Then \(c_1(F)=c_2(F)=0\) and \(F\not\simeq\mathcal O_{\PP^2}^{\oplus2}\). By the deformation results of
\cite[Remark~4.6]{Stromme1983} and
\cite[Proposition~4]{Beauville2018}, there exists, after shrinking
$B$, a vector bundle $\mathcal F$ on $\PP^2\times B$ such that
\[
\mathcal F_t\simeq\mathcal O_{\PP^2}^{\oplus2}
\quad(t\neq0),
\qquad
\mathcal F_0\simeq F.
\]

Set \(\mathscr F
\doteq
\mathcal F^\vee(-3)\). Then \(\mathscr F_t
\simeq
\mathcal O_{\PP^2}(-3)^{\oplus2}\) for \(t\neq0\), whereas dualizing and twisting
\eqref{eq:special-nonsplit-family} gives
\[
0
\longrightarrow
\mathcal O_{\PP^2}(-2)
\longrightarrow
\mathscr F_0
\longrightarrow
\mathcal I_p(-4)
\longrightarrow0.
\]
Hence
\[
\indeg(\mathscr F_t)=3
\quad(t\neq0),
\qquad
\indeg(\mathscr F_0)=2.
\]

The bundle $F(3)$ is globally generated and \(h^1(F(3))=0\). The same vanishing holds on the nearby trivial fibers. Therefore,
after shrinking $B$, cohomology and base change gives a vector bundle \(\mathcal H=q_*\mathcal F(3)\) whose fiber at $t$ is $H^0(\mathcal F_t(3))$. Choose three relative
sections whose specializations at $0$ are general. After shrinking
$B$ once more, their evaluation morphism
\[
\ev:
\mathcal O_{\PP^2\times B}^{\oplus3}
\longrightarrow
\mathcal F(3)
\]
has rank two away from a finite subscheme on every fiber.

Dualizing gives
\[
\Psi:
\mathscr F
\longrightarrow
\mathcal O_{\PP^2\times B}^{\oplus3}.
\]
For every $t$, this is injective with torsion-free rank-one cokernel,
hence
\[
0
\longrightarrow
\mathscr F_t
\longrightarrow
\mathcal O_{\PP^2}^{\oplus3}
\longrightarrow
\mathcal I_{Z_t}(6)
\longrightarrow0
\]
for some zero-dimensional $Z_t\subset\PP^2$. Since \(c_1(\mathscr F_t)=-6\) and \(
c_2(\mathscr F_t)=9\), the Chern class formula gives \(\length(Z_t)=27\). By Lemma~\ref{lem:flatness-finite-cokernel}, the cokernel of $\Psi$
is flat over $B$, and hence the family defines
\[
\gamma:B\longrightarrow Z_{27}.
\]
The displayed splitting types give
\[
\gamma(t)\in I_3^{27}
\quad(t\neq0),
\qquad
\gamma(0)\in I_2^{27},
\]
as claimed.
\end{Example}

\section{Gradient triples and integrable gaps}\label{sec:gradient-gaps}

This section investigates how the gradient locus
\(
G_d\subset\PP(W_d)
\)
(see Proposition~\ref{prop:parameter-space-gradient}) intersects the strata $Z_j$ and $I_e^j$, with particular emphasis on
the free case.

Since $G_d$ is a linear subspace, it is smooth of dimension
\[
\binom{d+3}{2}-1.
\]
Let \(p\in G_d\cap Z_j\) be an unobstructed point. The local dimension inequality gives
\begin{align*}
\dim_p(G_d\cap Z_j)
&\geq
\dim_p Z_j+\dim G_d-\dim\PP(W_d)\\
&=
\left(
3\binom{d+2}{2}-j-1
\right)
+
\left(
\binom{d+3}{2}-1
\right)
-
\left(
3\binom{d+2}{2}-1
\right)\\
&=
\binom{d+3}{2}-j-1.
\end{align*}
At a transverse point, equality holds.

In the free case, putting together the dimension formulas, we obtain lower bounds for the dimension of free gradient strata and free arrangement strata, in the next result.

\begin{Proposition}\label{prop:dimension-bound-free-families}
Let
\[
1\leq e\leq\left\lfloor\frac d2\right\rfloor,
\qquad
j=d(d-e)+e^2.
\]
Then:

\begin{itemize}
\item[(a)] Every irreducible component of \(I_e^j\cap G_d\) has dimension at least
\[
\begin{cases}
e^2-de+d+3e+3, & e<d/2,\\[1mm]
-\dfrac14d^2+\dfrac52d+2, & e=d/2.
\end{cases}
\]

\item[(b)] Every irreducible component of \(
I_e^j\cap A_d\) has dimension at least
\[
\begin{cases}
e^2-de-\dfrac12d^2+\dfrac12d+3e+3,
& e<d/2,\\[1mm]
-\dfrac34d^2+2d+2,
& e=d/2.
\end{cases}
\]
\end{itemize}

At every transverse point, the local dimension of the corresponding intersection is equal to the displayed lower bound. Consequently, every irreducible component containing a transverse point has exactly this dimension.

Both intersections are nonempty. In the balanced case $e=d/2$, the
gradient intersection is nowhere transverse for every even
$d\geq12$, while the arrangement intersection is nowhere transverse
for every even $d\geq4$.
\end{Proposition}

\begin{proof}
Set \(j_e=d(d-e)+e^2\) and write \(
\mathbb P=\mathbb P(W_d)\). Then
\[
\dim \mathbb P
=
3\binom{d+2}{2}-1.
\]
Let $C$ be an irreducible component of
$I_e^{j_e}\cap G_d$, and let $\eta$ be its generic point.
Since $I_e^{j_e}$ is locally closed in $\mathbb P$, there exists an
open neighborhood $\Omega\subseteq\mathbb P$ of $\eta$ such that
$I_e^{j_e}\cap\Omega$ is closed in $\Omega$. Applying the standard
dimension inequality for intersections in the smooth ambient variety
$\Omega$ (see \cite[I, Theorem~7.2]{hartshorne2013algebraic}) and
localizing at $\eta$, we obtain
\[
\dim C
\geq
\dim I_e^{j_e}
+
\dim G_d
-
\dim\mathbb P.
\]
The same argument applies with $A_d$ in place of $G_d$.

By Proposition~\ref{prop:free-strata-smooth}, the free stratum is
smooth and has dimension
\[
\dim I_e^{j_e}
=
\begin{cases}
d^2-de+e^2+3d+3e+3,
& e<d/2,\\[1mm]
\displaystyle
\frac34d^2+\frac92d+2,
& e=d/2.
\end{cases}
\]

On the other hand, Proposition~\ref{prop:parameter-space-gradient}
gives \(\dim G_d=\binom{d+3}{2}-1\). Therefore, if $e<d/2$,
\begin{align*}
\dim C
&\geq
d^2-de+e^2+3d+3e+3-d(d+2)\\
&=
e^2-de+d+3e+3,
\end{align*}
while in the balanced case $e=d/2$,
\begin{align*}
\dim C
&\geq
\frac34d^2+\frac92d+2-d(d+2)\\
&=
-\frac14d^2+\frac52d+2.
\end{align*}
This proves~(a).

For the arrangement locus, we have \(\dim A_d=2(d+1)\). Thus, for every irreducible component $C'$ of
$I_e^{j_e}\cap A_d$,
\[
\dim C'
\geq
\dim I_e^{j_e}
+
2(d+1)
-
\left(3\binom{d+2}{2}-1\right).
\]
If $e<d/2$, this becomes
\begin{align*}
\dim C'
&\geq
d^2-de+e^2+3d+3e+3
+2d+2
-\left(3\binom{d+2}{2}-1\right)\\
&=
e^2-de-\frac12d^2+\frac12d+3e+3.
\end{align*}
If $e=d/2$, we obtain
\begin{align*}
\dim C'
&\geq
\frac34d^2+\frac92d+2
+2d+2
-\left(3\binom{d+2}{2}-1\right)\\
&=
-\frac34d^2+2d+2.
\end{align*}
This proves~(b).

At a point at which the corresponding intersection is transverse,
the tangent spaces span the tangent space of the ambient projective
space. Consequently the local dimension of the intersection is the
expected one
\[
\dim I_e^{j_e}+\dim X-\dim\mathbb P,
\qquad
X=G_d\ \text{or}\ A_d.
\]
In particular, any irreducible component containing a transverse
point has dimension equal to the corresponding lower bound above.

Finally, the intersections under consideration are nonempty.
Indeed, by
\cite[Theorems~1.1 and~1.2]{dimca2015exponents}, for every
\[
1\leq e\leq\left\lfloor\frac d2\right\rfloor
\]
there exist free line arrangements of degree $d+1$ with exponents
$(e,d-e)$. Hence
\[
I_e^{j_e}\cap A_d\neq\emptyset \text{ and }I_e^{j_e}\cap G_d\neq\emptyset.
\]

It remains to consider the balanced case $e=d/2$. For the gradient
locus, the expected dimension is
\[
-\frac14d^2+\frac52d+2
=
\frac{-d^2+10d+8}{4}.
\]
The roots of the quadratic polynomial above are \( d=5\pm\sqrt{33}\), so the formula gives a negative value for every even $d\geq12$. Since the intersection is nevertheless nonempty, it cannot be
transverse at any of its points.

Similarly, for the arrangement locus the expected dimension is
\[
-\frac34d^2+2d+2
=
\frac{-3d^2+8d+8}{4},
\]
whose roots are \(
d=\frac{4\pm2\sqrt{10}}{3}\). Thus it is negative for every even $d\geq4$, and hence
$I_{d/2}^{j_e}\cap A_d$ is nowhere transverse in these degrees.
\end{proof}

\begin{Remark}\label{rmk:known-free-families-of-divisors}
Let us see how our framework (especially Proposition~\ref{prop:free-strata-smooth} and Corollary~\ref{cor:free-points-picture}) relates to the following known families of free divisors on $\mathbb{P}^2$:
\begin{itemize}
    \item[(a)] The \emph{Cayley sextics} (studied in \cite{simis2006depth}):
    $$
    f_{a,b}(x,y,z) = 4(x^2 + b y^2 + a xz)^3 - 27 a^2 (x^2 + b y^2)^2 z^2
    $$
    where $(a,b) \in \mathbb{A}^2_k$ such that $ab \neq 0$. For each $f = f_{a,b}$ gradient ideal, the syzygy sheaf is $\T_{J_f} \simeq \mathcal{O}_{\mathbb{P}^2}(-2) \oplus \mathcal{O}_{\mathbb{P}^2}(-3)$.

    Here $d=5$, $e=2$, and \(j=d(d-e)+e^2=19\). By Corollary~\ref{cor:free-points-picture}, the stratum $I_2^{19}$ is smooth and $Z_{19}$ is smooth along it. Moreover,
    \[
    \dim I_2^{19}=43.
    \]
    The intersection \(G_5\cap I_2^{19}\) has dimension at least $8$, whereas the family $f_{a,b}$ is only two-dimensional, so it does not fill an irreducible component of the intersection.
    
    \item[(b)] The \emph{Simis-Tohaneanu} family (see \cite[Proposition 2.2]{simis2014homology}): for $D = d+1 \geq 5$, the polynomials
    $$
    g_{a,b,c,t}(x,y,z) = y^{D-1}z + a x^D+ b x^2 y^{D-2} + c x y^{D-1} + t y^D
    $$
    for $a b \neq 0$. For each $g = g_{a,b,c,t}$, the syzygy sheaf is $\T_{J_g} \simeq \mathcal{O}_{\mathbb{P}^2}(-2) \oplus \mathcal{O}_{\mathbb{P}^2}(2-d)$. Hence, one has $e = 2$, so $j = (D-1)^2 - 2(D-3)$. 
    
    For $d = 4$, $e = 2 = d/2$ and $j = 12$. The whole parameter space for $d = 4$ is $\mathbb{P}(W_4) \simeq \mathbb{P}^{44}$. The stratum $I_2^{12}$ has dimension $32$, and the gradient locus $G_4 \cap I_2^{12}$ has dimension at least $8$, so this family also does not fill up a component of the intersection. 
    
    \item[(c)] The \emph{Nanduri} family (see \cite[Theorem 2.3]{nanduri2015family}): for $D = d+1 \geq 5$ integer and $v = \lfloor \frac{D}{2} \rfloor$, the polynomials 
    $$
    h_{F_1, F_2}(x,y,z) = x^{D-\alpha}F_1(x,y) + y^{v+\alpha+1} F_2(x,y) + x^\beta y^{D-\beta-1} z
    $$
    with $\alpha, \beta \geq 0$ and $0 \leq \alpha + \beta \leq \lfloor \frac{D+1}{2} \rfloor - 3$, $F_1(x,y)$ is square-free, homogeneous of degree $\alpha$ with $x \nmid F_1$ and $y \nmid F_1$, and $F_2(x,y)$ is homogeneous of degree $D-v-\alpha-1$ with $x \nmid F_2$ and $y \nmid F_2$. For each $h = h_{F_1, F_2}$, the syzygy sheaf is 
    $$
    \T_{J_h} \simeq \mathcal{O}_{\mathbb{P}^2}(-v)\oplus \mathcal{O}_{\mathbb{P}^2}(v-d),
    $$
    hence $e =  \min\{v, d-v\} = \lfloor d/2 \rfloor$. In particular, the initial degree is independent of the parameters
$\alpha$ and $\beta$.
\end{itemize}

Let us compare the dimension of the Nanduri family with the
dimension of the corresponding free stratum.  For fixed $(\alpha,\beta)$, the parameters are precisely the
coefficients of $F_1$ and $F_2$. The two corresponding summands
\[
x^{D-\alpha}F_1(x,y)
\qquad\text{and}\qquad
y^{v+\alpha+1}F_2(x,y)
\]
have disjoint monomial supports, while the coefficient of
\[
x^\beta y^{D-\beta-1}z
\]
has been normalized to be $1$. Hence two choices of $(F_1,F_2)$ give
the same projective polynomial only when they coincide. The
square-freeness and non-divisibility are open conditions, so
the Nanduri family for fixed $(\alpha,\beta)$ has dimension
\[
\begin{aligned}
N_{\alpha,\beta}
&=(\alpha+1)+(D-v-\alpha)\\
&=D-v+1\\
&=
D-\left\lfloor\frac D2\right\rfloor+1.
\end{aligned}
\]
In particular, for fixed $D$ this dimension is independent of
$(\alpha,\beta)$.

Since every member of the Nanduri family is free, it belongs to $
I_e^{d(d-e)+e^2}$ for $e=\lfloor\frac d2\rfloor$. By Corollary~\ref{cor:free-points-picture}, every such point is a
smooth point of the corresponding stratum $Z_{d(d-e)+e^2}$.

Notice that $N_{\alpha,\beta}$ grows linearly with $D$, whereas both
the free stratum $I_e^j$ and the gradient locus $G_d$ have quadratic
dimension in $D$. For orientation, the first cases are:
\[
\begin{array}{c|c|c|c|c}
D & (d,e,j)
& \dim I_e^j
& \dim_{\mathrm{exp}}(G_d\cap I_e^j)
& N_{\alpha,\beta}
\\ \hline
5 & (4,2,12) & 32 & 8 & 4\\
6 & (5,2,19) & 43 & 8 & 4\\
7 & (6,3,27) & 56 & 8 & 5
\end{array}
\]
For $D=5,6$ one necessarily has $\alpha+\beta=0$, whereas for
$D=7$ the possibilities $\alpha+\beta=0,1$ yield the same splitting
type and hence lie in the same free stratum.

More generally, let first \(D=2m+1\). Then
\(
d=2m\), \(e=m\), and \(N_{\alpha,\beta}=m+2\), while Proposition~\ref{prop:dimension-bound-free-families} gives
\[
\dim_{\mathrm{exp}}(G_d\cap I_e^j)
=
-m^2+5m+2.
\]
Therefore \(N_{\alpha,\beta}
-
\dim_{\mathrm{exp}}(G_d\cap I_e^j)
=
m(m-4)\).

If instead \(D=2m\), then \(
d=2m-1\), \(e=m-1\) and \(N_{\alpha,\beta}=m+1\), while
\[
\dim_{\mathrm{exp}}(G_d\cap I_e^j)
=
-m^2+6m-1.
\]
Hence \(N_{\alpha,\beta}
-
\dim_{\mathrm{exp}}(G_d\cap I_e^j)
=
m^2-5m+2\).

Consequently, for every even $D\geq10$ and every odd $D\geq11$, the Nanduri family has dimension strictly larger than the expected
dimension of $G_d\cap I_e^j$. Every irreducible component containing it is therefore non-transverse.
\end{Remark}

In this final remark, we give a geometric interpretation of \cite[Question 2.7]{JMNRS2026}.

\begin{Remark}\label{rmk:gaps-integrable-geometry}
For fixed $d$, the absence of integrable gaps is equivalent to
\[
I_e^j\neq\emptyset
\quad\Longrightarrow\quad
I_e^j\cap G_d\neq\emptyset.
\]

For the free strata this implication always holds. Indeed, by
\cite[Theorem~1.2]{dimca2015exponents}, for every
$1\le e\le\lfloor d/2\rfloor$ there exists a free arrangement of
$d+1$ lines with exponents $(e,d-e)$. Hence
\[
I_e^{d(d-e)+e^2}\cap A_d\neq\emptyset,
\]
and therefore also
\(
I_e^{d(d-e)+e^2}\cap G_d\neq\emptyset.
\)

The question of integrable gaps is therefore genuinely relevant for
the non-free strata. Even when $G_d\cap Z_j$ is nonempty, it may
happen that this intersection is contained entirely in the union of
some of the locally closed strata $I_e^j$, while missing others.

For instance, suppose that $Z_j$ is irreducible and let $e_0$ be its
generic initial degree. Then $I_{e_0}^j$ contains a dense open subset
of $Z_j$. Nevertheless, it is still possible a priori that \( G_d\cap I_{e_0}^j=\emptyset\), since the closed linear space $G_d$ could meet $Z_j$ entirely inside
the proper closed subset \(
Z_j\setminus I_{e_0}^j\). Thus density of $I_{e_0}^j$ in $Z_j$ alone does not rule out an
integrable gap.
\end{Remark}

\section{Syzygy bundles, semistability, and refined du Plessis--Wall bounds}\label{sec:semistable-syzygies}

The goal of this section is to study the syzygy bundle \(\T_J\) associated with a triple $[J]\in I_e^j$ from the point of view of
stability on $\PP^2$. We first use slope semistability directly,
without invoking a moduli space, to obtain a stronger numerical upper
bound for $j$ in the semistable range. Remarkably, this bound agrees,
for gradient triples, with the refined du Plessis--Wall bound.

We then consider the usual Gieseker moduli space of sheaves. In the
strictly slope-stable range $e>d/2$, the bundles $\T_J$ are also
Gieseker stable, and the universal syzygy family therefore induces a
natural morphism from $I_e^j$ to the corresponding moduli space. The
balanced case $e=d/2$ requires separate treatment, since slope
semistability does not in general imply Gieseker semistability.

\begin{Lemma}\label{lem:mu-semistability}
Let $d\geq1$, $1\leq j\leq d^2$, and let
$[J]\in I_e^j$. Then
\[
\T_J \text{ is $\mu$-semistable}
\iff
e\geq\left\lceil\frac d2\right\rceil,
\]
and
\[
\T_J \text{ is $\mu$-stable}
\iff
e>\frac d2.
\]
\end{Lemma}

\begin{proof}
Since \(c_1(\T_J)=-d\), we have \(\mu(\T_J)=-d/2\). A nonzero section of $\T_J(l)$ induces a rank-one subsheaf
\[
\mathcal O_{\PP^2}(-l)\hookrightarrow\T_J.
\]
Conversely, after saturation, any rank-one destabilizing subsheaf
produces such a section. Hence $\T_J$ is $\mu$-semistable precisely
when \(H^0(\T_J(l))=0\) for every \( l < \frac d2\), and it is $\mu$-stable precisely when
\(
H^0(\T_J(l))=0\) for every \(l\leq d/2\).
Since \( e=\indeg(\T_J)=
\min\{l:H^0(\T_J(l))\neq0\}
\), the assertions follow.
\end{proof}

\begin{Proposition}[Semistability and the refined upper bound]
\label{prop:semistable-refined-bound}
Let $[J]\in I_e^j$ and suppose \(
e\geq \lceil d/2\rceil\). Then
\begin{equation}\label{eq:refined-upper-general-triples}
j
\leq
\frac{d(d+1)}2+(d-1)e-e^2.
\end{equation}
Equivalently,
\begin{equation}\label{eq:refined-upper-dpw-form}
j
\leq
d(d-e)+e^2
-
\binom{2e-d+1}{2}.
\end{equation}

Consequently, in the semistable range the possible values of $j$ are
contained in
\begin{equation}\label{eq:semistable-refined-interval}
I^{\mathrm{ss}}_{d,e}
\doteq
\left[
d(d-e),\,
d(d-e)+e^2-\binom{2e-d+1}{2}
\right].
\end{equation}
\end{Proposition}

\begin{proof}
Set \(E\doteq\T_J\). By Lemma~\ref{lem:mu-semistability}, $E$ is $\mu$-semistable. By the
definition of $e$,
\[
H^0(E(e-1))=0.
\]

We claim also that \(H^2(E(e-1))=0\). Indeed, by Serre duality,
\[
H^2(E(e-1))
\simeq
H^0(E^\vee(-e-2))^\vee.
\]
Since $E$ is $\mu$-semistable, so is
$E^\vee(-e-2)$, and
\[
\mu(E^\vee(-e-2))
=
\frac d2-e-2<0.
\]
A semistable sheaf of negative slope has no nonzero global sections,
and therefore
\[
0=H^0(E^\vee(-e-2)) = H^2(E(e-1))^\vee.
\]

It follows that \(\chi(E(e-1))=-h^1(E(e-1))\leq0\). Since
\[
\rank(E)=2,\qquad
c_1(E)=-d,\qquad
c_2(E)=d^2-j,
\]
Riemann--Roch on $\PP^2$ gives
\[
\chi(E(t))
=
t^2+(3-d)t
+
2-\frac{3d}{2}-\frac{d^2}{2}+j.
\]
Putting $t=e-1$, we obtain
\[
\chi(E(e-1))
=
j+e^2-(d-1)e-\frac{d(d+1)}2.
\]
Thus
\[
j
\leq
\frac{d(d+1)}2+(d-1)e-e^2,
\]
which proves~\eqref{eq:refined-upper-general-triples}. Finally,
\[
\frac{d(d+1)}2+(d-1)e-e^2
=
d(d-e)+e^2
-
\binom{2e-d+1}{2},
\]
giving~\eqref{eq:refined-upper-dpw-form}.
\end{proof}

\begin{Remark}[Relation with the refined du Plessis--Wall bound]
\label{rmk:refined-dpw-semistability}
For gradient triples in the range \(e> d/2\), the estimate~\eqref{eq:refined-upper-dpw-form} is precisely the
refined du Plessis--Wall upper bound
\cite[Theorem~3.2]{CTC-Plessis}. Indeed, for a reduced plane curve
of degree $D=d+1$, with
\[
e=\indeg(\syz(J_F))
\qquad\text{and}\qquad
j=\tau(F),
\]
their bound becomes
\[
j
\leq
d(d-e)+e^2-\binom{2e-d+1}{2}.
\]
When $d$ is even and $e=d/2$, the same inequality remains valid, but
\[
\binom{2e-d+1}{2}=0,
\]
so one recovers the ordinary du Plessis--Wall upper bound.

Thus Proposition~\ref{prop:semistable-refined-bound} gives a
vector-bundle interpretation of the correction term
\[
\binom{2e-d+1}{2}.
\]
For three-generated ideals the same numerical restriction
follows from slope semistability of the syzygy bundle together with
the vanishings 
\[
h^0(\T_J(e-1))=h^2(\T_J(e-1))=0.
\]

In this sense, the du Plessis--Wall correction term arises as a
cohomological consequence of semistability of the syzygy bundle.
\end{Remark}

\begin{Proposition}\label{prop:morphism-gieseker}
Let $d\geq1$, $j\geq0$, and suppose \(
e> d/2\). Let \(M^{\mathrm{Gs}}_{d,j}\) denote the moduli space of Gieseker-stable rank-two torsion-free sheaves $E$ on $\PP^2$ with $c_1(E) = -d$ and $c_2(E) = d^2-j$.

Then the universal syzygy bundle induces a natural morphism
\[
\Sigma:
I_e^j
\longrightarrow
M^{\mathrm{Gs}}_{d,j},
\qquad
[J]\longmapsto[\T_J].
\]

If $d$ is odd, this construction applies to the entire
$\mu$-semistable range
\[
e\geq\left\lceil\frac d2\right\rceil.
\]
\end{Proposition}

\begin{proof}
Let \(q:\PP^2\times I_e^j\lar I_e^j\) be the projection and let $\mathscr T$ be the universal syzygy
bundle. Since $I_e^j\subset Z_j$ and the universal quotient over
$Z_j$ is flat, $\mathscr T$ is flat over $I_e^j$. Its fibers have
fixed Chern classes \(c_1=-d\) and \(c_2=d^2-j\).

By Lemma~\ref{lem:mu-semistability}, the inequality $e>d/2$ implies that every fiber $\T_J$ is $\mu$-stable. Hence every fiber is Gieseker stable. Therefore $\mathscr T$ is a
flat family of stable sheaves with fixed Hilbert polynomial. By the
standard construction and universal property of the coarse moduli
space of Gieseker-semistable sheaves
(see, for example, \cite[Chapter~4]{huybrechts2010geometry}),
the family determines a morphism
\[
\Sigma:I_e^j\longrightarrow M^{\mathrm{Gs}}_{d,j}.
\]

If $d$ is odd, then \(e\geq \lceil d/2 \rceil\)
automatically implies $e>d/2$, proving the last assertion.
\end{proof}

\begin{Remark}[The balanced even case]
\label{rmk:balanced-gieseker}
Suppose \(
d=2m\) and \(e=m\). Then $\T_J$ is $\mu$-semistable but not necessarily Gieseker semistable. Indeed, set \(
F\doteq\T_J(m)\). Then
\[
c_1(F)=0 \text{ and }H^0(F)\neq0.
\]

The minimality of $e$ implies that a nonzero section of $F$ has no
divisorial zero. Hence it gives an exact sequence
\[
0
\longrightarrow
\mathcal O_{\PP^2}
\longrightarrow
F
\longrightarrow
\mathcal I_B
\longrightarrow0,
\]
where $B\subset\PP^2$ is zero-dimensional and
\[
\deg(B)
=
c_2(F)
=
\frac{3d^2}{4}-j.
\]

If $\deg(B)>0$, then
\[
P_F(t)
=
P_{\mathcal O_{\PP^2}}(t)
+
P_{\mathcal I_B}(t)
=
2P_{\mathcal O_{\PP^2}}(t)-\deg(B).
\]
Since $\rank(F)=2$, its reduced Hilbert polynomial is
\[
p_F(t)
=
\frac{P_F(t)}2
=
P_{\mathcal O_{\PP^2}}(t)
-\frac{\deg(B)}2.
\]
Thus
\(
p_{\mathcal O_{\PP^2}}(t)>p_F(t)\), so the subsheaf
\(\mathcal O_{\PP^2}\subset F\) Gieseker-destabilizes $F$.

At the extremal value \(
j=3d^2/4\), one has $B=\emptyset$, so
\[
0
\longrightarrow
\mathcal O_{\PP^2}
\longrightarrow
F
\longrightarrow
\mathcal O_{\PP^2}
\longrightarrow0.
\]
Since \(H^1(\mathcal O_{\PP^2})=0\),
the sequence splits, and therefore
\[
\T_J
\simeq
\mathcal O_{\PP^2}(-d/2)^{\oplus2}.
\]
Thus, among the balanced points with \(e=d/2\), the extremal value
\[
j=\frac{3d^2}{4}
\]
is the unique case in which $\T_J$ is Gieseker semistable.
\end{Remark}

\begin{Corollary}\label{cor:exceptional-sheaves}
Let $d\geq3$ and set
\[
j_{\max}(d)
=
\begin{cases}
\dfrac{3d^2}{4},
& d\ \mathrm{even},\\[2mm]
\dfrac{3d^2-3}{4},
& d\ \mathrm{odd}.
\end{cases}
\]
Suppose
\[
e\geq\left\lceil\frac d2\right\rceil
\qquad\text{and}\qquad
j=j_{\max}(d).
\]
Then:

\begin{itemize}
    \item[(a)] If $d$ is even, necessarily \(
    e=d/2\) and \(\T_J
    \simeq
    \mathcal O_{\PP^2}(-d/2)^{\oplus2}.
    \)

    \item[(b)] If $d$ is odd, necessarily \(
    e=(d+1)/2\), and the corresponding extremal stable bundle is \(
    \T_J
    \simeq
    T_{\PP^2}\left(-(d+3)/2\right).
    \)
\end{itemize}
\end{Corollary}

\begin{proof}
Suppose first that $d=2m$. Write $e=m+s$ with $s\geq0$. By Proposition~\ref{prop:semistable-refined-bound}, \(j \leq 3m^2-s(s+1)\).
Since
\[
j=j_{\max}(d)=3m^2,
\]
we must have $s=0$, hence $e=m=d/2$. The conclusion then follows
from Remark~\ref{rmk:balanced-gieseker}.

Now suppose that $d=2m+1$ and write
\[
e=m+1+s,
\qquad s\geq0.
\]
The same bound gives
\[
j
\leq
3m(m+1)-s(s+2).
\]
Since \(
j=j_{\max}(d)=3m(m+1)\), we obtain $s=0$, so
\[
e=m+1=\frac{d+1}{2}.
\]

Set \(F\doteq\T_J(m+1)\). A direct Chern class computation gives
\[
c_1(F)=1,
\qquad
c_2(F)=1.
\]
Moreover, \(h^0(F)\neq0\) and \(h^0(F(-1))=0\). Thus a nonzero section of $F$ has no divisorial zero and determines
an exact sequence
\[
0\longrightarrow
\mathcal O_{\PP^2}
\longrightarrow
F
\longrightarrow
\mathcal I_p(1)
\longrightarrow0
\]
for some point $p\in\PP^2$. 

Since $F$ is locally free, this extension is non-split. On the other
hand,
\[
\Ext^1(\mathcal I_p(1),\mathcal O_{\PP^2})\simeq k,
\]
so there is, up to isomorphism, a unique nontrivial locally free
extension of this form. The bundle
$T_{\PP^2}(-1)$ realizes this extension. Hence \(F\simeq T_{\PP^2}(-1)\), and therefore
\[
\T_J
\simeq
T_{\PP^2}(-m-2)
=
T_{\PP^2}\left(-\frac{d+3}{2}\right).
\]
\end{proof}

\section{The quartic strata of fixed Bourbaki degree}\label{sec:quartic-strata}

In this section, we specialize to equigenerated triples with $d=3$,
with particular emphasis on the gradient locus $G_3$ arising from
quartic plane curves. In Theorem~\ref{thm:strata-d-3}, we prove that every nonempty stratum $Z_j$ is smooth and irreducible and determine its decomposition into initial-degree strata $I_e^j$.

Let \(X=V(F)\subseteq\PP^2\) be a reduced singular quartic. We assume throughout this section that
$X$ is not a cone, equivalently that the three partial derivatives
$F_x,F_y,F_z$ are $k$-linearly independent. Thus
\[
J_F=(F_x,F_y,F_z)
\]
is minimally generated by three cubic forms and has height two. The
excluded reduced cone is a union of four concurrent lines; its
gradient ideal is generated by only two cubic forms and therefore
does not belong to the three-generated setting considered here.

For $d=3$, \(\PP(W_3)=\PP^{29}\) and \(G_3\simeq\PP(R_4)=\PP^{14}\) by Proposition~\ref{prop:parameter-space-gradient}. We denote by
\[
B_{3,b}\subseteq G_3\simeq\PP^{14}
\]
the quartic locus of fixed Bourbaki degree $b$.
The numerical classification
\cite[Theorem~4.1]{JMNRS2026}, together with Wall's classification
of reduced singular quartics~\cite{Wall-quartics}, determines the
possible Bourbaki degrees in terms of the singularity configuration
and the initial Jacobian syzygy degree.

The only singularity configurations for which the Bourbaki degree is
not determined by the configuration alone are $3A_1$ and $A_3$;
in those cases the additional datum is
\[
e(X)=\indeg(\syz(J_F)).
\]

If the same ADE configuration occurs in different geometric realizations, we distinguish the corresponding families by superscripts
\[
\mathrm{irr},\qquad (3+1),\qquad (2+2),\qquad (2+1+1),\qquad (1+1+1+1),
\]
according to whether the quartic is irreducible, a cubic plus a line, two conics, a conic plus two lines, or four lines. Thus, for instance, $4A_1^{(2+2)}$ denotes the family of quartics which are unions of two conics meeting transversely in four points.

The following result describes the quartic Bourbaki strata. For all
but two singularity configurations, the Bourbaki degree is determined
by the configuration alone. The only ambiguity occurs for
\[
3A_1
\qquad\text{and}\qquad
A_3,
\]
for which the Bourbaki degree may be $4$ or $6$ according to the
initial degree \(e(X)=\indeg(\syz(J_F))\).

\begin{Theorem}\label{quartic-Bour-strata}
The quartic Bourbaki strata satisfy the following properties.

\begin{enumerate}
\item[\rm(i)] One has
\[
B_{3,5}=\varnothing,
\qquad
B_{3,b}\neq\varnothing \quad \text{for } b\in\{0,1,2,3,4,6,7,8\}.
\]

\item[\rm(ii)] For $b\in\{0,1,2,3,7,8\}$, the locus $B_{3,b}$ is a finite union of irreducible locally closed families, all of the same dimension. These families are the $\PGL_3$-saturations of the corresponding Wall normal-form families listed in Table~\ref{tab:quartic-Bour-strata}. In particular, all locally closed irreducible families appearing in $B_{3,b}$ have the dimension indicated in the table.

\item[\rm(iii)]
The union $B_{3,4}\cup B_{3,6}$ is supported on the four
$11$-dimensional equisingular families
\[
3A_1^{\mathrm{irr}},\qquad
3A_1^{(3+1)},\qquad
(A_1+A_2)^{\mathrm{irr}},\qquad
A_3^{\mathrm{irr}}.
\]
More precisely,
\[
B_{3,4}
=
\left\{
X:
\Sigma_X\in
\{3A_1^{\mathrm{irr}},
  3A_1^{(3+1)},
  A_3^{\mathrm{irr}}\},
\ e(X)=2
\right\},
\]
whereas
\[
B_{3,6}
=
\left\{
X:
\Sigma_X\in
\{3A_1^{\mathrm{irr}},A_3^{\mathrm{irr}}\},
\ e(X)=3
\right\}
\sqcup
(A_1+A_2)^{\mathrm{irr}}.
\]
In particular, the Bourbaki degree is not determined by the
singularity configuration only for the configurations
$3A_1$ and $A_3$.
\end{enumerate}
\end{Theorem}

\begin{table}[ht]
\caption{Quartic strata of fixed Bourbaki degree for the non-exceptional values.}
\centering
\renewcommand{\arraystretch}{1.15}
{\small
\begin{tabular}{c|c|p{9.5cm}}
\hline
$b$ & dimension & Irreducible locally closed families in $B_{3,b}$ \\
\hline
$8$ & $13$ &
$A_1^{\mathrm{irr}}$ \\[0.4ex]
\hline
$7$ & $12$ &
$2A_1^{\mathrm{irr}}$,\ $A_2^{\mathrm{irr}}$ \\[0.4ex]
\hline
$3$ & $10$ &
$(2A_1+A_2)^{\mathrm{irr}}$,\ $(A_1+A_3)^{\mathrm{irr}}$,\ $D_4^{\mathrm{irr}}$,\ $2A_2^{\mathrm{irr}}$,\ $A_4^{\mathrm{irr}}$,\ $(A_1+A_3)^{(3+1)}$,\ $4A_1^{(3+1)}$,\ $4A_1^{(2+2)}$ \\[0.4ex]
\hline
$2$ & $9$ &
$(A_1+2A_2)^{\mathrm{irr}}$,\ $(A_2+A_3)^{\mathrm{irr}}$,\ $(A_1+A_4)^{\mathrm{irr}}$,\ $A_5^{\mathrm{irr}}$,\ $D_5^{\mathrm{irr}}$,\ $A_5^{(3+1)}$,\ $(2A_1+A_3)^{(3+1)}$,\ $(2A_1+A_3)^{(2+2)}$,\ $(A_1+D_4)^{(3+1)}$,\ $(3A_1+A_2)^{(3+1)}$,\ $(A_1+D_4)^{(2+1+1)}$,\ $5A_1^{(2+1+1)}$ \\[0.4ex]
\hline
$1$ & $8$ &
$3A_2^{\mathrm{irr}}$,\ $(A_2+A_4)^{\mathrm{irr}}$,\ $A_6^{\mathrm{irr}}$,\ $E_6^{\mathrm{irr}}$,\ $(A_1+A_5)^{(3+1)}$,\ $(A_1+A_2+A_3)^{(3+1)}$,\ $(A_1+D_5)^{(3+1)}$,\ $D_6^{(3+1)}$,\ $2A_3^{(2+2)}$,\ $(A_1+A_5)^{(2+2)}$,\ $(3A_1+A_3)^{(2+1+1)}$,\ $(2A_1+D_4)^{(2+1+1)}$,\ $6A_1^{(1+1+1+1)}$ \\[0.4ex]
\hline
$0$ & $7$ &
$(A_2+A_5)^{(3+1)}$,\ $A_7^{(2+2)}$,\ $(A_1+2A_3)^{(2+1+1)}$,\ $(A_1+D_6)^{(2+1+1)}$,\ $(3A_1+D_4)^{(1+1+1+1)}$,\ $E_7^{(3+1)}$ \\[0.4ex]
\hline
\end{tabular}}
\label{tab:quartic-Bour-strata}
\end{table}

\begin{proof}
By Wall's classification of reduced singular quartic curves, refined
by decomposition type in the reducible case, the quartic locus is
decomposed into finitely many normal-form families. Each such family
is irreducible, and its $\PGL_3$-saturation is again irreducible.
For every family appearing in the classification, the explicit
normal-form parameter count gives dimension
\[
14-\tau(\Sigma),
\]
where $\tau(\Sigma)$ denotes the sum of the local Tjurina numbers of
the singularities occurring in the configuration $\Sigma$.
Since all singularities involved are simple, this agrees with the
expected equianalytic codimension. We emphasize that the global
dimension statement here comes from Wall's normal-form
classification, rather than merely from the absence of local analytic
moduli.

\cite[Theorem~4.1]{JMNRS2026} shows that the only possible Bourbaki values are
\[
0,1,2,3,4,6,7,8,
\]
and that the value $5$ never occurs. This proves \rm(i).

Let $b\in\{0,1,2,3,7,8\}$. For these values, \cite[Theorem~4.1]{JMNRS2026} shows that $\Bour(X)=b$ is determined by the singularity configuration. Hence $B_{3,b}$ is exactly the union of the equisingular families with the configurations listed in Table~\ref{tab:quartic-Bour-strata}. In each row of the table, all configurations have the same total Tjurina number. Therefore, all irreducible locally closed families appearing in $B_{3,b}$ have the same dimension $14-\tau(\Sigma)$, namely the value indicated in the second column. This proves \rm(ii).

Finally, \cite[Theorem~4.1]{JMNRS2026} shows that the only
singularity configurations for which the Bourbaki degree is not
determined by the singularity type are
\[
3A_1
\qquad\text{and}\qquad
A_3.
\]
The configuration $A_1+A_2$, on the other hand, necessarily has
initial degree $e=3$, and hence Bourbaki degree $6$.

Refining by decomposition type gives the families appearing in
part~(iii). Since all of them have total Tjurina number $3$, the
Bourbaki formula becomes
\[
\Bour(X)=e^2-3e+6,
\]
so that
\[
\Bour(X)=4 \quad\text{for }e=2,
\qquad
\Bour(X)=6 \quad\text{for }e=3.
\]
The refined quartic classification then gives precisely the
decomposition of $B_{3,4}$ and $B_{3,6}$ stated above.
\end{proof}

\begin{Remark}
The stratum $B_{3,8}$ is the dense open stratum in the discriminant
of quartic plane curves and corresponds to quartics with exactly one node. The value $5$ is a genuine gap in the quartic Bourbaki spectrum.
The only singularity configurations for which the Bourbaki degree is
not determined by the configuration alone are $3A_1$ and $A_3$; in
these cases one must additionally specify the initial degree of the
Jacobian syzygy module.
\end{Remark}

\begin{Lemma}\label{lem:d-3-generic-smoothness}
For every $0\leq j\leq7$, the scheme $Z_j$ is smooth. Moreover,
the morphism
\[
\Phi_j:Z_j\longrightarrow\Hilbert^j(\PP^2)
\]
is smooth.
\end{Lemma}

\begin{proof}
Let $[J]=[f_1,f_2,f_3]\in Z_j$, and let \(Z=V(J)\subset\PP^2\). We first observe that
\begin{equation}\label{eq:no-four-secant-cubic-base}
\deg(Z\cap L)\leq3
\end{equation}
for every line $L\subset\PP^2$. Indeed, suppose that \(
\deg(Z\cap L)\geq4\). For each $i=1,2,3$, the restriction
\[
\restr {f_i} {L}\in H^0(L,\mathcal O_L(3))
\]
vanishes on the zero-dimensional scheme $Z\cap L$ of length at
least four. Since a nonzero section of
$\mathcal O_{\PP^1}(3)$ has a zero divisor of degree $3$, it follows
that
\[
\restr {f_i} {L}=0.
\]
Hence the equation of $L$ divides each of $f_1,f_2,f_3$, contradicting
the assumption that the triple has no common divisorial factor.

Now suppose, by contradiction, that \(
H^1(\mathcal I_Z(3))\neq0\). Since
\[
\deg Z=j\leq7=2\cdot3+1,
\]
the postulation lemma
\cite[Lemma~34]{BernardiGimiglianoIda2011} implies that there exists a line $L\subset\PP^2$ such that
\[
\deg(Z\cap L)\geq3+2=5.
\]
This contradicts \eqref{eq:no-four-secant-cubic-base}. Therefore \(
H^1(\mathcal I_Z(3))=0\) for every $[J]\in Z_j$ and every $0\leq j\leq7$. Lemma~\ref{lem:obstruction-cohomology} now gives both the smoothness
of $Z_j$ and the smoothness of
\[
\Phi_j:Z_j\longrightarrow\Hilbert^j(\PP^2).
\]
\end{proof}

\begin{Corollary}\label{cor:d3-Zj-smooth-irreducible}
For every $0\leq j\leq7$, the scheme $Z_j$ is smooth and
irreducible. Moreover,
\[
\dim Z_j=29-j.
\]
\end{Corollary}

\begin{proof}
By Lemma~\ref{lem:d-3-generic-smoothness}, every $Z_j$ is smooth and $\Phi_j$ is smooth. For $j=0$, nonemptiness follows from the
existence of smooth quartics, while for $1\leq j\leq7$ it follows
from Theorem~\ref{quartic-Bour-strata}.

Hence Lemma~\ref{lem:irreducibility-from-Hilbert-map} shows that
$Z_j$ is irreducible for every $0\leq j\leq7$. Finally,
Lemma~\ref{lem:dimension-Z-j-smooth-points} gives
\[
\dim Z_j
=
3h^0(\mathcal O_{\PP^2}(3))-j-1
=
29-j.
\]
\end{proof}

\begin{Theorem}\label{thm:strata-d-3}
For $d=3$, the nonempty strata are precisely
\[
Z_0,\ldots,Z_7.
\]
Each $Z_j$ is smooth and irreducible of dimension \(\dim Z_j=29-j\).

Moreover, for every $0\leq j\leq7$, the gradient locus $G_3$ meets
$Z_j$ properly in $\PP(W_3)$:
\[
\dim(G_3\cap Z_j)
=
14-j
=
\dim G_3+\dim Z_j-\dim\PP(W_3).
\]
The decomposition of $Z_j$ into initial-degree strata and the
corresponding quartic gradient loci are given in
Table~\ref{table:geom-strata-d-3}.
\begin{table}[ht]
\caption{Geometry of the strata of $3$-equigenerated ideals.}
\label{table:geom-strata-d-3}
\centering
\begin{tabular}{|c|l|c|c|l|}
\hline
$j$ & strata & smooth? & irreducible? & gradient quartics\\
\hline
$0$ & $I_3^0$
    & yes & yes & smooth quartics\\
\hline
$1$ & $I_3^1$
    & yes & yes & $B_{3,8}=A_1$\\
\hline
$2$ & $I_3^2$
    & yes & yes & $B_{3,7}=2A_1\sqcup A_2$\\
\hline
$3$ & $I_2^3\sqcup I_3^3$, with $I_3^3$ dense open
    & yes & yes & $B_{3,4}\sqcup B_{3,6}$\\
\hline
$4$ & $I_2^4$
    & yes & yes & $B_{3,3}$\\
\hline
$5$ & $I_2^5$
    & yes & yes & $B_{3,2}$\\
\hline
$6$ & $I_1^6\sqcup I_2^6$, with $I_2^6$ dense open
    & yes & yes & $B_{3,1}$\\
\hline
$7$ & $I_1^7$
    & yes & yes & $B_{3,0}$ (free quartics)\\
\hline
$\geq8$ & $\varnothing$
    & -- & -- & --\\
\hline
\end{tabular}
\end{table}
\end{Theorem}

\begin{proof}
By Corollary~\ref{cor:d3-Zj-smooth-irreducible}, for every \(0\leq j\leq7\), the scheme $Z_j$ is smooth and irreducible, with \(
\dim Z_j=29-j\).

We first determine the possible initial degrees. Since $d=3$, the
general numerical bounds give
\[
\begin{array}{c|c}
e & \text{possible values of }j\\
\hline
1 & 6\leq j\leq7,\\
2 & 3\leq j\leq6,\\
3 & 0\leq j\leq3.
\end{array}
\]
For $e=2,3$, Proposition~\ref{prop:semistable-refined-bound}
specializes to the ranges displayed above. Combining these with the
basic Bourbaki bounds gives exactly the decompositions
\[
\begin{aligned}
Z_0&=I_3^0,\\
Z_1&=I_3^1,\\
Z_2&=I_3^2,\\
Z_3&=I_2^3\sqcup I_3^3,\\
Z_4&=I_2^4,\\
Z_5&=I_2^5,\\
Z_6&=I_1^6\sqcup I_2^6,\\
Z_7&=I_1^7.
\end{aligned}
\]

The required nonemptiness is also explicit. Smooth quartics give
$Z_0\neq\emptyset$, while Theorem~\ref{quartic-Bour-strata} gives
gradient examples for every $1\leq j\leq7$. In the mixed stratum
$Z_3$, the loci $B_{3,4}$ and $B_{3,6}$ realize respectively the
initial degrees $e=2$ and $e=3$. In the mixed stratum $Z_6$,
Example~\ref{ex:subscheme-flat-and-quot-flat} realizes both
$I_1^6$ and $I_2^6$.

Since $Z_3$ and $Z_6$ are irreducible, Proposition~
\ref{prop:generic-initial-degree} shows that the strata of maximal
initial degree, \(I_3^3\subset Z_3\) and \(I_2^6\subset Z_6\) are dense open subsets.

We next show that no further $Z_j$ occur. Suppose that $[J]=[f_1,f_2,f_3] \in U_3 \subset \mathbb{P}(W_3)$ has zero-dimensional base scheme $Z$, and let
\[
V=\langle f_1,f_2,f_3\rangle\subset H^0(\mathcal O_{\PP^2}(3)).
\]
Two general members $F,G\in V$ have no common component, and hence
define a complete intersection
\[
\Gamma=V(F,G)
\]
of type $(3,3)$ and length $9$. Since $Z\subseteq\Gamma$, one has
$\length(Z)\leq9$.

If $\length(Z)=9$, then $Z=\Gamma$, and therefore
\[
V\subseteq H^0(\mathcal I_\Gamma(3))
=
\langle F,G\rangle,
\]
contradicting $\dim V=3$. If $\length(Z)=8$, the residual scheme of
$Z$ in $\Gamma$ has length one. By the Cayley--Bacharach theorem for
the $(3,3)$ complete intersection $\Gamma$, every cubic through $Z$
also passes through the residual point. Hence again
\[
H^0(\mathcal I_Z(3))
=
H^0(\mathcal I_\Gamma(3))
=
\langle F,G\rangle,
\]
contradicting the fact that it contains the three-dimensional space
$V$. Thus \(Z_j=\varnothing\) for \(j\geq8\).

It remains to describe the intersection with the gradient locus.
A point of $G_3\cap U_3$ corresponds to a reduced non-cone quartic
$X=V(F)$, and
\[
\deg V(J_F)=\tau(F).
\]
Hence $G_3\cap Z_j$ is precisely the locus of reduced non-cone
quartics of total Tjurina number $j$.

For $j=0$, this is the open locus of smooth quartics in \(\PP(R_4)\simeq\PP^{14}\),
and therefore has dimension $14$. For $1\leq j\leq7$,
Theorem~\ref{quartic-Bour-strata} together with Wall's classification
shows that every irreducible locally closed family occurring in
$G_3\cap Z_j$ has dimension
\(14-j\). Thus, for every $0\leq j\leq7$,
\[
\dim(G_3\cap Z_j)=14-j.
\]

Since
\[
\dim G_3=14,
\qquad
\dim Z_j=29-j,
\qquad
\dim\PP(W_3)=29,
\]
we have
\[
14-j
=
\dim G_3+\dim Z_j-\dim\PP(W_3).
\]
Therefore $G_3$ and $Z_j$ meet properly for every $0\leq j\leq7$.
\end{proof}

\textbf{Declaration of generative AI and AI-assisted technologies in the manuscript preparation process}

During the preparation of this work, the author used Chatgpt for revision organization. The author reviewed and edited the output as needed and take full responsibility for the content of the published article. 



\end{document}